\title{Sums, Differences and Dilates}
\author{
Jonathan Cutler \and Luke Pebody \and Amites Sarkar
}
\date{\today}

\documentclass[11pt]{article}
\usepackage{amssymb}
\usepackage{amsmath}
\usepackage{amsthm}
\usepackage{bbm}
\usepackage{mathtools}
\usepackage[margin=1in]{geometry}
\newcommand{\Z}{\mathbb{Z}}
\newcommand{\R}{\mathbb{R}}
\newcommand{\E}{\mathbb{E}}

\newcommand{\Prb}{\mathbb{P}}

\newcommand{\Var}{\mathrm{Var}}
\newcommand{\mult}{\mathrm{Mult}}

\newtheorem{thm}{Theorem}
\newtheorem{lem}[thm]{Lemma}
\newtheorem{clm}[thm]{Claim}
\newtheorem{cor}[thm]{Corollary}
\newtheorem{qn}[thm]{Question}

\newtheorem{conj}[thm]{Conjecture}

\theoremstyle{definition}

\newtheorem*{defn}{Definition}

\DeclarePairedDelimiter{\norm}{\lVert}{\rVert}

\newcommand{\set}[1]{\left\{{#1}\right\}}
\newcommand{\setof}[2]{\left\{{#1}\,:\,{#2}\right\}}

\AtBeginDocument{%
   \def\MR#1{}
}

\begin{document}
\maketitle

\begin{abstract}
Given a set of integers $A$ and an integer $k$, write $A+k\cdot A$ for the set
$\{a+kb:a\in A,b\in A\}$. Hanson and Petridis~\cite{HPSource}
showed that if $|A+A|\le K|A|$ then $|A+2\cdot A|\le K^{2.95}|A|$ and also that $|A+2\cdot A|\le(K|A|)^{4/3}$.
We present a new construction, the {\em Hypercube+Interval construction}, which lies close to these upper bounds
and which shows in particular that, for all $\epsilon>0$, there exist $A$ and $K$
with $|A+A|\le K|A|$ but with $|A+2\cdot A|\ge K^{2-\epsilon}|A|$.

Further, we analyse a method of Ruzsa~\cite{RuzsaRandom}, and generalise it to
give fractional analogues of the sizes of sumsets, difference sets and dilates. We apply this method
to a construction of Hennecart, Robert and Yudin~\cite{HRY} to prove that, for all $\epsilon>0$, there exists a
set $A$ with $|A-A|\ge |A|^{2-\epsilon}$ but with $|A+A|<|A|^{1.7354+\epsilon}$.

The second author would like to thank E. Papavassilopoulos for useful
discussions about how to improve the efficiency of his computer searches.
\end{abstract}

\section{Introduction and Definitions}\label{S:Introduction}

The study of the size of the sumset $|A+A|$ and difference set $|A-A|$
(sometimes denoted $DA$)
in terms of $|A|$ is a central theme in additive
combinatorics. For instance, {\it Freiman's theorem}
states that if $|A+A|\le K|A|$, then $A$ must be a large fraction of a {\it
generalized arithmetic progression}, and the {\it Balog-Szemer\'edi-Gowers
theorem} states that if $A$ has large {\it additive energy}, then $A$ must
contain a large subset $A'$ such that $|A'+A'|/|A'|$ is small.
For the precise definitions and statements, we refer the reader to~\cite{Zhao}.

For a finite set $A\subset\Z$, with $|A|=n$, we have that
\[
|A+A|\le \frac{n(n+1)}{2}{\rm \ \ \ and\ \ \ }|A-A|\le n^2-n+1,
\]
with equality in both cases precisely when $A$ is a {\it Sidon set},
that is, a set
containing no nontrivial {\it additive quadruple} $(a,b,c,d)\in A^4$
with $a+b=c+d$
(and consequently no nontrivial $(a,b,c,d)$ with $a-b=c-d$). In other words,
if $|A+A|$ is as large as it can possibly be, then so is $|A-A|$, and
conversely. In 1992,
Ruzsa~\cite{RuzsaRandom} showed, using an ingenious probabilistic construction,
that $|A+A|$ can be small, while $|A-A|$ can be almost as large as possible,
and vice-versa.  In particular, he showed the following.
\begin{thm}[Ruzsa, 1992]
	For every large enough $n$, there is a set $A$ such that $|A|=n$ with
	\[
		|A+A|\leq n^{2-c}\qquad\text{and}\qquad |A-A|\geq n^{2}-n^{2-c},
	\]
	where $c$ is a positive absolute constant.
	Also, there is a set $B$ with $|B|=n$,
	\[
		|B-B|\leq n^{2-c}\qquad\text{and}\qquad |B+B|\geq \frac{n^2}2-n^{2-c}.
	\]
\end{thm}

A few years later, Hennecart, Robert and Yudin~\cite{HRY} constructed a set $A$
of size $n$ with $|A+A|\sim n^{1.4519}$ but $|A-A|\sim n^{1.8462}$. Their
construction was inspired by convex geometry, specifically the {\it difference
body inequality} of Rogers and Shephard~\cite{RogShep}. In the other direction,
the study and classification of {\it MSTD sets} (sets with more sums than
differences) began with Conway in 1967, and has now attracted a large literature
(see~\cite{IMLZ} for a recent survey).

Note that Hennecart, Robert and Rudin in fact constructed a set $A\subset \Z^d$.
But any such construction in $\Z^d$ can be easily translated to a construction in $\Z$ using
an appropriately chosen {\it Freiman homomorphism}, i.e., a map $\phi$ of the form
\[
\phi(x_1,\ldots,x_d)=\lambda_1x_1+\cdots+\lambda_dx_d,
\]
for an appropriate choice of integers $\lambda_i$.

Another line of investigation was opened by Bukh~\cite{Bukh} in 2008.
Given a set of integers $A$ and an integer $k$, define the {\em dilate set} $A+k\cdot A$ by
\[
A+k\cdot A=\{a_1+ka_2:a_1, a_2\in A\}.
\]
For $k=1$, this is just the sumset, $A+A$, and for $k=-1$ it is the
difference set, $A-A$.
Note that, for example, the dilate set $A+2\cdot A$ is generally
a strict subset of $A+A+A$, where each of the three summands can be distinct.

Bukh proved many results on general sums of dilates
$\lambda_1\cdot A+\cdots+\lambda_k\cdot A$
(for arbitrary integers $\lambda_1,\ldots,\lambda_k$), including lower and
upper bounds on their sizes. Some of these results were phrased in terms of sets
with {\it small doubling}, namely, sets $A\subset\Z$ with $|A+A|\le K|A|$, for
some fixed constant $K$ (known as the {\it doubling constant}).
For such a set $A$,
{\it Pl\"unnecke's inequality}~\cite{Plu} (see also~\cite{Pet}) shows that
\[
|A+2\cdot A|\le|A+A+A|\le K^3|A|,
\]
and Bukh asked if the exponent 3 could be improved. This question was
answered affirmatively in 2021 by Hanson and Petridis~\cite{HPSource},
who proved the following.
\begin{thm}[Hanson-Petridis, 2021]\label{thm:hp1}
	If $A\subset \Z$ and $|A+A|\leq K|A|$, then
\[
	|A+2\cdot A|\le K^{2.95}|A|.
\]
\end{thm}
\noindent They were also able to prove a result that improves Theorem~\ref{thm:hp1} when
$K$ is large.
\begin{thm}[Hanson-Petridis, 2021]\label{thm:hp2}
If $A\subset \Z$ and $|A+A|\leq K|A|$, then
\[
	|A+2\cdot A|\le(K|A|)^{4/3}.
\]
\end{thm}

\noindent Our contributions in this paper are best understood in the context of
{\it feasible regions} of the plane, and so we make the following definition.
\begin{defn}
For fixed integers $k$ and $l$,
we define the \emph{feasible region} $F_{k,l}$ to be the closure of the
set $E_{k,l}$ of {\it attainable points}
\[
E_{k,l}=
\left\{\left(
	\frac{\log|A+k\cdot A|}{\log|A|},
	\frac{\log|A+l\cdot A|}{\log|A|}
\right)\right\},
\]
as $A$ ranges over finite sets of integers.
\end{defn}
Note that for any $k$ and $l$, we have that $E_{k,l}\subset [1,2]^2$, which
follows from the fact that $|A|\le|A+k\cdot A|\le |A|^2$ for all $A$.  For every
$k$ and $l$, each $A$ produces a point in $E_{k,l}$. With $A$ fixed,
we can generate a sequence of sets, indexed by the dimension $d$, by taking a
Cartesian product $A^d\subset\Z^d$, and then we will have
$|A^d+k\cdot A^d|=|A+k\cdot A|^d$. The advantage of the logarithmic measure we
are using is
that all examples in this sequence, generated from the same set $A$, correspond
to the same point $(x,y)\in E_{k,l}$. Another useful fact is that the set
$F_{k,l}$ is convex. We prove this in
Section~\ref{S:Feasible Regions}.

The first series of results in this paper concerns the size of the dilate set
$A+2\cdot A$. We present a construction, the {\it Hypercube+Interval
construction}, which improves all previous bounds, and is close to the above upper
bounds of Hanson and Petridis. Specifically, this construction shows that the graph
of the piecewise-linear function
\[
y=\min\left(2x-1,(\log_34)x\right)=
\begin{cases}
2x-1&1\le x\le\log_{\frac94}3=1.3548\ldots\\
(\log_34)x&\log_{\frac94}3\le x\le 2\\
\end{cases}
\]
is entirely contained in $F_{1,2}$. This will allow us to prove a partial
converse to Theorem~\ref{thm:hp1}, namely that
for any $\epsilon>0$, there exists a positive constant $K$ and a set $A$
with $|A+A|\le K|A|$ but $|A+2\cdot A|\ge K^{2-\epsilon}|A|$. Thus the true
bound here is between 2 and 2.95.  We also give some negative results, showing
that neither Sidon Sets, nor
subsets of $\{0,1\}^d\subset\Z^d$, can give rise to feasible points outside the
regions already proved feasible. Finally, we give a lower bound for the region
$F_{1,2}$, which is an easy consequence of Pl\"unnecke's inequality.
All these bounds and constructions are illustrated in Figure~\ref{fig:feas}.

Our next series of results concerns the relationship between the sizes of $A+A$
and $A-A$, and thus relates
to the feasible region $F_{1,-1}$. This is one of the oldest topics in additive
combinatorics, with results
going back to Freiman and Pigarev~\cite{FP} and Ruzsa~\cite{Ruz1976} in the
1970s (and indeed Conway in the 1960s).

Our starting point is a 1992 paper of Ruzsa~\cite{RuzsaRandom}. Ruzsa
constructed sets $A\subset \Z^d$ for which $A-A$ is very large, but $A+A$ is very small,
in the following way. Start with a finite set $S\subset\Z$ with $|S+S|<|S-S|$.
Then, for a fixed probability $0<q<1$, select a random subset $A\subset\Z^d$ by
taking each element of $S^d$ independently with probability $q^d$. For an
appropriate choice of $q$, this ``boosts" the discrepancy between $|S+S|$
and $|S-S|$ enough to prove the first part of Theorem 1. The second part is
proved in a similar way.

In Section 4, we analyse and generalise Ruzsa's method
from~\cite{RuzsaRandom}, leading to the following concept, which can be seen
as a continuous analogue of the size of a sumset and that of a dilate.

\begin{defn}
A {\em fractional dilate $\gamma$} is a map $\gamma:\Z\to\R^{+}\cup\{0\}$ with
finite support ${\rm supp}(\gamma)$. We define the {\em size of a fractional dilate} to be
\[
\norm{\gamma}=\inf_{0\le p\le1}\sum_{n\in{\rm supp}(\gamma)}\gamma(n)^p.
\]
A {\em fractional set} is a fractional dilate $\alpha$ for which
$\alpha(n)\le 1$ for all $n\in \Z$.
\end{defn}

Note that, if $\alpha$ is a fractional set,
then $\norm{\alpha}=\sum_{n\in\Z}\alpha(n)$, i.e., the above infimum is attained at $p=1$.
On the other hand, if $\gamma$ is a dilate for which $\gamma(n)\ge 1$ for all $n\in\Z$, then
$\norm{\gamma}=|{\rm supp}(\gamma)|$, and the infimum is attained at $p=0$. In general, we describe
a fractional dilate as being {\em opulent}, {\em spartan} or {\em p-comfortable}
if the above infimum is attained at $p=0$, $p=1$ or $0<p<1$ respectively.
Theorem~\ref{T:fractionalsize} gives a useful alternative characterization of fractional dilates.

We can identify an actual subset $S$ of $\Z$ with the fractional set $\mathbbm{1}_S$, which is
both spartan and opulent. For any such sets $S$ and $T$, $\mathbbm{1}_S+k\cdot\mathbbm{1}_T$ will be
opulent, so that
\[
\norm{\mathbbm{1}_S+k\cdot\mathbbm{1}_T}=|S+k\cdot T|.
\]

Given a fractional set $\alpha$, let us say that a random set
$S_n\subseteq\Z^{n}$ is {\em drawn from $\alpha^n$} if each element of
$\Z^{n}$ is chosen independently, and the probability that
$(i_1, i_2, \ldots, i_n)$ is selected is
$\alpha(i_1)\alpha(i_2)\ldots\alpha(i_n)$. Moreover, for fractional sets
$\alpha, \beta$ and an integer $k$, let $\alpha+k\cdot\beta$ denote the
fractional dilate defined by the formula
\[
(\alpha+k\cdot\beta)(n)=\sum_{\substack{(i,j)\\i+kj=n}}\alpha(i)\beta(j).
\]

\noindent The point of these definitions is the following pair of theorems, which we prove in Section~\ref{S:Dilates 2}.

\begin{thm}\label{T:fractional}
Let $\alpha$ be a fractional set with $\norm{\alpha}> 1$, and suppose
$S_n\subseteq\Z^{n}$ is drawn from $\alpha^n$. Then
\[
\E|S_n|=\norm{\alpha}^n\ \ \ \ \  \Var|S_n|\le\norm{\alpha}^n
\]
and
\[
\lim_{n\to\infty}(\E|S_n+k\cdot S_n|)^{1/n}=\norm{\alpha+k\cdot\alpha}.
\]
\end{thm}
\begin{thm}\label{T:spartan}
Let $\alpha$ be a fractional set with $\norm{\alpha}> 1$, and suppose
$S_n\subseteq\Z^{n}$ is drawn from $\alpha^n$. If $\alpha+k\cdot\alpha$ is ``strictly spartan" (see later
for an explanation of this terminology), in the sense that
\[
\sum_{n\in{\rm supp}(\gamma)}\gamma(n)\log\gamma(n)<0,
\]
then, with probability tending to 1,
\begin{align*}
|S_n+k\cdot S_n|&\ge\tfrac12|S_n|^2\textrm{ if }k\ne 1\\
|S_n+k\cdot S_n|&\ge\tfrac14|S_n|^2\textrm{ if }k=1.
\end{align*}
\end{thm}

\noindent From now on, the phrase ``with high probability", abbreviated to {\bf whp}, means ``with
probability tending to 1 as the dimension (usually denoted by $n$) tends to infinity".

In Section~\ref{S:RuzsaHRY}, we apply these theorems to a construction of
Hennecart, Robert and Yudin~\cite{HRY}, to construct a fractional set $\alpha$
for which $\alpha-\alpha$ is spartan, but $\alpha+\alpha$ is not.

\begin{thm}\label{T:1.7354}
There exists a fractional set $\alpha$ for which $\norm\alpha>1$,
$\alpha-\alpha$ is strictly spartan (so that $\norm{\alpha-\alpha}=\norm{\alpha}^2$), and
$\norm{\alpha+\alpha}\le\norm{\alpha}^{1.7354}$.
\end{thm}

\noindent This will allow us to prove that $(1.7354, 2)$ is feasible for $F_{1,-1}$.

\begin{cor}\label{C:1.7354}
For all $\epsilon>0$, there exists a finite subset $A\subseteq\Z$ such that
$|A-A|\ge|A|^{2-\epsilon}>1$ but $|A+A|\le|A|^{1.7354+\epsilon}$.
\end{cor}

\begin{proof}
Let $\alpha$ be the fractional set with properties as in
Theorem~\ref{T:1.7354}, and let $S_n$ be drawn from $\alpha^n$.

\noindent First, from Theorem~\ref{T:fractional} and Chebyshev's inequality
\[
\Prb(||S_n|-\norm\alpha^n|>0.1\norm\alpha^n)
\le 100\Var|S_n|/\norm\alpha^{2n}
\le 100/\norm\alpha^n\to 0,
\]
so that with high probability
\[
||S_n|-\norm\alpha^n|\le 0.1\norm\alpha^n.\eqno(1)
\]
Second, since $\alpha-\alpha$ is strictly spartan, Theorem~\ref{T:spartan} shows that
with high probability
\[
|S_n-S_n|\ge\tfrac12|S_n|^2.\eqno(2)
\]
Finally, Theorem~\ref{T:fractional} shows that
\[
\lim_{n\to\infty}\E|S_n+S_n|^{1/n}\to \norm{\alpha+\alpha}=\norm{\alpha}^{1.7354}.
\]
Consequently, for all $\epsilon>0$, we will have
\[
\E|S_n+S_n|\le\norm{\alpha}^{(1.7354+\epsilon)n}
\]
for all sufficiently large $n$, and for such $n$
\[
\Prb(|S_n+S_n|>\norm{\alpha}^{(1.7354+2\epsilon)n})\to 0
\]
by Markov's inequality. Invoking (1), we have that for sufficiently large $n$
\[
\Prb(|S_n+S_n|>|S_n|^{1.7354+3\epsilon})\to 0,
\]
so that with high probability
\[
|S_n+S_n|\le |S_n|^{1.7354+3\epsilon}.\eqno(3)
\]
The conclusion of the corollary follows from (1), (2) and (3).
\end{proof}

In the other direction, it follows from results of Freiman and Pigarev~\cite{FP}
and Ruzsa~\cite{Ruz1976} that $(x,2)$ is not attainable for any
$x<3/2$. All these results are illustrated in Figure~\ref{fig:feas2}.
\begin{figure}[h]
	\begin{center}
		\includegraphics[scale=.3]{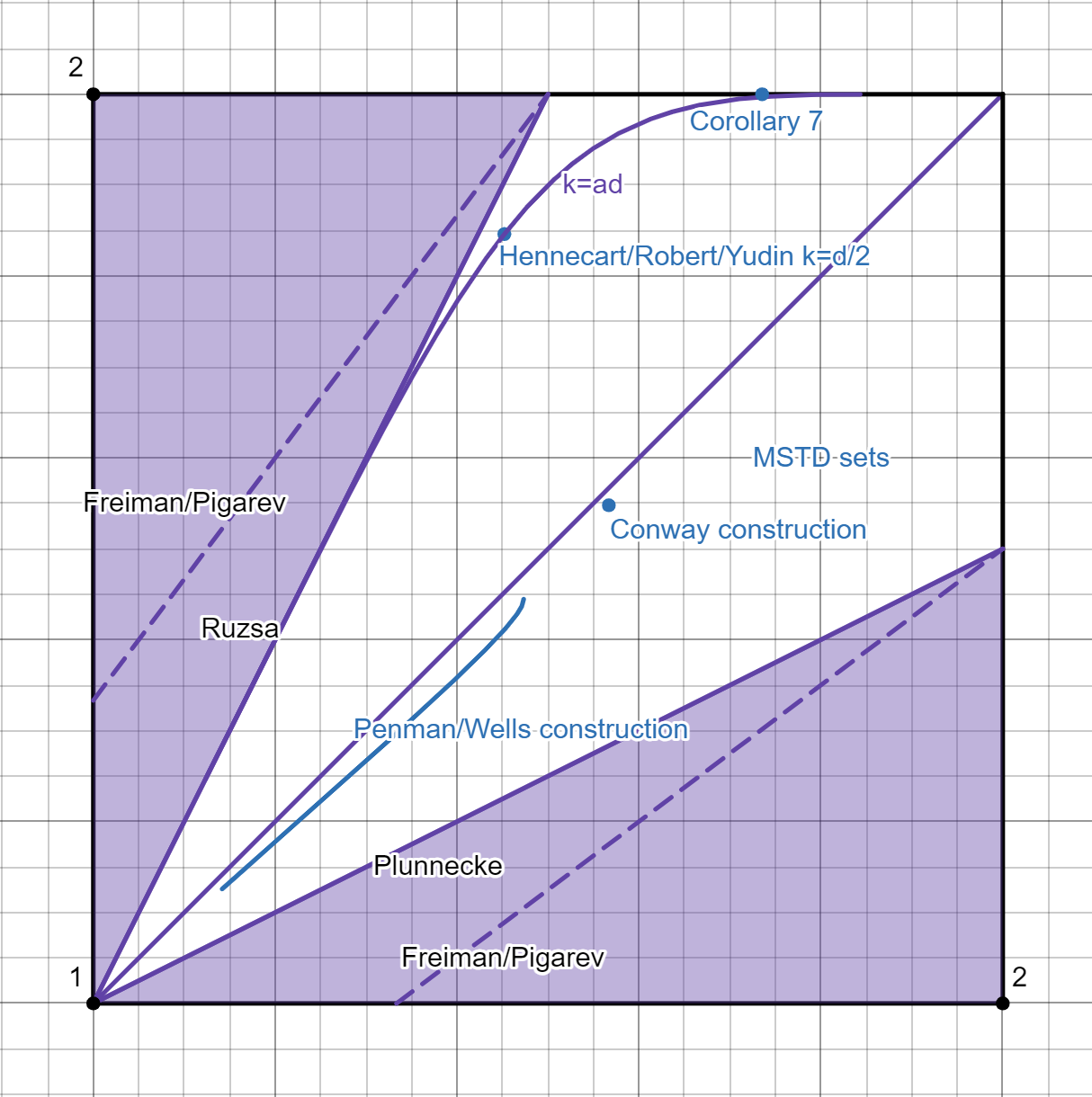}
		\caption{The feasible region $F_{1,-1}$}\label{fig:feas2}
	\end{center}
\end{figure}

Finally, in Section~\ref{S:Open Questions}, we discuss many open questions about
$F_{1,-1}$ and $F_{1,2}$,
and about feasible regions in general.

\section{Feasible Regions}\label{S:Feasible Regions}

We remind the reader of the definition of a feasible region. For fixed integers
$k$ and $l$, the {\it feasible region} $F_{k,l}$ is defined as the closure of
the set $E_{k,l}$ of {\it attainable points}
\[
E_{k,l}=
\left\{\left(
	\frac{\log|A+k\cdot A|}{\log|A|},
	\frac{\log|A+l\cdot A|}{\log|A|}
\right)\right\}\subset [1,2]^2,
\]
as $A$ ranges over finite sets of integers.  Note once again that the inclusion
follows from the fact that $|A|\le|A+k\cdot A|\le |A|^2$ for all such $A$.
Since $E_{k,l}\subset [1,2]^2$, we have also $F_{k,l}\subset [1,2]^2$.  As
mentioned in the introduction, we now prove that $F_{k,l}$ is convex.

\begin{thm}\label{T:basicbutnecessary}
For all nonzero $k,l$, the feasible region $F_{k,l}$ is convex, and contains the
diagonal $D=\{(x,x):1\le x\le 2\}$.
\end{thm}

\begin{proof}
First we prove the convexity. To do this, we first consider points
$(x,y),(x',y')\in E_{k,l}$, and take $t\in[0,1]$. We will show that
$(tx+(1-t)x',ty+(1-t)y')\in F_{k,l}.$
Since $(x,y)\in E_{k,l}$, there exists a set $A\subset\Z$ with
\[
|A+k\cdot A|=|A|^x{\rm\ and\ }|A+l\cdot A|=|A|^y.
\]
Likewise, since $(x',y')\in E_{k,l}$, there exists a set $B\subset\Z$ with
\[
|B+k\cdot B|=|B|^{x'}{\rm\ and\ }|B+l\cdot B|=|B|^{y'}.
\]
Setting $\beta=\log|B|/\log|A|$, choose a sequence $q_1,q_2,\ldots$ of rational
numbers such that
\[
\lim_{i\to\infty} q_i=\frac{t\beta}{1-t+t\beta}.
\]
For each such $q_i=r/s$, we consider a set $A_i\subset\Z^s$, defined as
\[
A_i=\underbrace{A\times A\times\cdots\times A}_r\times
\underbrace{B\times B\times\cdots\times B}_{s-r},
\]
in which there are $r$ factors of $A$ and $s-r$ factors of $B$. We have
\begin{align*}
|A_i|&=|A|^r|B|^{s-r},\\
|A_i+k\cdot A_i|&=|A|^{rx}|B|^{(s-r)x'}, \text{ and}\\
|A_i+l\cdot A_i|&=|A|^{ry}|B|^{(s-r)y'}.
\end{align*}
Therefore,
\begin{align*}
\frac{\log|A_i+k\cdot A_i|}{\log|A_i|}
&=\frac{rx\log|A|+(s-r)x'\log|B|}{r\log|A|+(s-r)\log|B|}\\
&=\frac{q_ix+(1-q_i)x'\beta}{q_i+(1-q_i)\beta},
\end{align*}
which tends to $tx+(1-t)x'$ as $i\to\infty$. Similarly,
\[
\frac{\log|A_i+l\cdot A_i|}{\log|A_i|}\to ty+(1-t)y',
\]
as $i\to\infty$. Consequently, $(tx+(1-t)x',ty+(1-t)y')\in F_{k,l}$.

Now, given points $(x,y),(x',y')\in F_{k,l}$, we may take sequences of points $
(x_j,y_j)$ and $(x'_j,y'_j)$ from $E_{k,l}$ tending to $(x,y)$ and $(x',y')$
respectively. For each $j$, the above argument shows that
\[
(tx_j+(1-t)x_j',ty_j+(1-t)y'_j)\in F_{k,l}.
\]
Consequently, letting $j\to\infty$, we have that
\[
(tx+(1-t)x',ty+(1-t)y')\in F_{k,l},
\]
and the convexity is proved.

To show that $(1,1)\in F_{k,l}$, we consider the set $A:=A_N=\{1,2,\ldots,N\}$
for $N\gg\max(k,l)$. We have
\[
|A+k\cdot A|=(k+1)(N-1)+1{\rm\ \ and\ \ }|A+l\cdot A|=(l+1)(N-1)+1,
\]
so that, as $N\to\infty$,
\[
\left(\frac{\log|A+k\cdot A|}{\log|A|},
\frac{\log|A+l\cdot A|}{\log|A|}\right)\to(1,1).
\]
To show that $(2,2)\in F_{k,l}$, let $b>\max(|k|,|l|)+1$, and consider the set
$B=\{1,b,b^2,\ldots,b^N\}$.
We have
\[
|B+k\cdot B|\ge\tfrac12|B|^2{\rm\ \ and\ \ }|B+l\cdot B|\ge\tfrac12|B|^2,
\]
so that, as $N\to\infty$,
\[
\left(\frac{\log|B+k\cdot B|}{\log|B|},
\frac{\log|B+l\cdot B|}{\log|B|}\right)\to(2,2).
\]
It now follows by convexity that $D=\{(x,x):1\le x\le 2\}\subset F_{k,l}$.
\end{proof}

This result easily generalises to higher dimensions.

\section{A construction for $F_{1,2}$}\label{S:Construction for (1,2)}

In this section, we present various results about the feasible region $F_{1,2}$.
In particular, as stated in the introduction, we give a partial converse to
a result of Hanson and Petridis (Theorem~\ref{thm:hp1}).

%If set $A$ is the union of sets $A_1, \ldots, A_n$, then it is clear that
%
%\begin{tabular}{r l c l}
%$\max_{1\le i\le n}|A_i|$ &$\le$ &$|A|$ &$\le\sum_{1\le i\le n}|A_i|$\\
%$\max_{1\le i\le j\le n}|A_i+A_j|$&$\le$ &
%$|A+A|$&$\le\sum_{1\le i\le j\le n}|A_i+A_j|,\quad\text{and}$\\
%$\max_{1\le i,j\le n}|A_i+2\cdot A_j|$&$\le$ &
%$|A+2\cdot A|$&$\le\sum_{1\le i,j\le n}|A_i+2\cdot A_j|$.
%\end{tabular}

Our construction, the {\em Hypercube + Interval construction}, is very simple. Let
\[
H_n=\left\{\sum_{i=0}^{n-1}a_i4^i:\ a_i\in\{0,1\}\right\},
\quad\quad
I_k=\left\{0,1,\ldots,\frac{4(4^{k-1}-1)}3\right\}
\quad\text{and}\quad
A_{n,k}=H_n\cup I_k.
\]
In other words, $H_n$ denotes the set of all natural numbers whose base $4$ representation
has length at most $n$ and contains only $0$s and $1$s (the hypercube), and $I_k$ is just
an interval.  We begin by giving bounds on the sizes of various sumsets and dilates related
to $H_n$ and $I_k$.

\begin{thm}~\label{T:sumsets}
For $n\ge k>\frac{n+1}2$, and with notation as above, we have:
\begin{align*}
|I_k|&=\frac{4^k-1}3\ge|H_n|=2^n\\
|H_n+H_n|&=3^n\\
|H_n+I_k|&=2^{n-k+1}\frac{4^{k}-1}3\ge|I_k+I_k|\\
|H_n+2\cdot H_n|&=4^n\ge \max\{|H_n+2\cdot I_k|,|I_k+2\cdot H_n|,|I_k+2\cdot I_k|\}.
\end{align*}
\end{thm}

\begin{proof}
%Since there are two choices for each $a_i$ where $0\leq i\leq n-1$, we have that
%$|H_n|=2^n$.  Also, since $(4^k-1)/3$ is an integer, we know that
%$|I_k|=(4^k-1)/3$. Further, since $k\ge\frac{n+2}2$, we have
%$4^k\ge 2^{n+2}>1+3\times 2^n$ and so $(4^k-1)/3>2^n$.

The first assertion is trivial. For $H_n+H_n$, note that
\[
H_n+H_n=\left\{\sum_{i=0}^{n-1}a_i4^i:\ a_i\in\{0,1,2\}\right\}.
\]
In other words, $H_n+H_n$ consists of the natural numbers whose base $4$
representation has length at most $n$ and contains only $0$s, $1$s, and $2$s.
Thus $|H_n+H_n|=3^n$. Similarly, $H_n+2\cdot H_n$ consists of those natural numbers
whose base $4$ representation has length at most $n$ and contains only 0s, 1s, 2s and 3s,
but this is just $\{0,1,\ldots,4^n-1\}$. Further, since $\max H_n\ge \max I_k$, all of
the sets $H_n+2\cdot I_k$, $I_k+2\cdot H_n$ and $I_k+2\cdot I_k$ are subsets of
$\{0,1,\ldots,4^n-1\}$, and are therefore of size at most $4^n$.

It remains to bound $|H_n+I_k|$. To this end, note that
\[
H_n=\{0, 1\}+\{0, 4\}+\{0,4^2\}+\cdots+\{0, 4^{n-1}\}.
\]
Also, if $p\ge q$ are positive integers,
then
\[
\{0,1,\ldots,p-1\}+\{0, q\}=\{0,1,\ldots,p+q-1\}.
\]
Consequently,
\begin{align*}
H_n+I_k&=I_k+\{0,1\}+\{0,4\}+\cdots+\{0,4^{k-1}\}+\{0,4^k\}+\cdots+\{0,4^{n-1}\}\\
&=\left\{0,1,\ldots,\frac{4(4^{k-1}-1)}3\right\}+\{0,1\}+\{0,4\}+\cdots+\{0,4^{k-1}\}+\{0,4^k\}+\cdots+\{0,4^{n-1}\}\\
&=\left\{0,1,\ldots,\frac{4(4^{k-1}-1)}3+1\right\}+\{0,4\}+\cdots+\{0,4^{k-1}\}+\{0,4^k\}+\cdots+\{0,4^{n-1}\}\\
& \qquad\vdots\\
&=\left\{0,1,\ldots,\frac{2(4^k-1)}3-1\right\}+\{0,4^k\}+\cdots+\{0,4^{n-1}\}\supseteq I_k+I_k.
\end{align*}
Moreover, the set $\{0, 4^k\}+\cdots+\{0, 4^{n-1}\}$ consists of multiples of
$4^k$, which are all further than $2\left(\frac{4^k-1}3\right)$ apart. It follows that $H_n+I_k$ consists
of $2^{n-k}$ intervals of length $2\left(\frac{4^k-1}3\right)$.
%
%We leave to the interested reader the job of calculating the sizes of
%$|I_k+I_k|$, $|H_n+2\cdot I_k|$, $|I_k+2\cdot H_n|$, and $|I_k+2\cdot I_k|$.
%\footnote{They are $2\frac{4^k-1}3-1$, $2^{n+1-k}(2\times 4^{k-1}-1)$ twice, and $4^k-3$ respectively.}
\end{proof}

\noindent Recall that $A_{n,k}=H_n\cup I_k$. Using Theorem~\ref{T:sumsets}, we can show the following.

\begin{cor}\label{C:technicallowerbound}
Fix $\alpha\in\left(\frac12,1\right)$, and set $k=\lfloor\alpha n\rfloor$. Then, as $n\to\infty$,
\begin{align*}
\frac{\log|A_{n,k}|}n&\to\alpha\log 4\\
\frac{\log|A_{n,k}+A_{n,k}|}n
&\to\max\left\{\log 3,\frac{1+\alpha}2\log 4\right\}\\
\frac{\log|A_{n,k}+2\cdot A_{n,k}|}n
&\to \log 4.
\end{align*}
\end{cor}
\begin{proof}
%Note that if $k$ is a fixed constant and $a_{n,i}$ for $1\le i\le k$ and
%$1\le n$ are positive integers and $c_i$ for $1\le i\le k$ are real numbers
%such that $\lim_{n\to\infty}\log a_{n,i}/n=c_i$ for all $i$,
%and integers $b_n$ satisfy that $\max_i a_{n,k}\le b_n\le\sum_i a_{n,k}$, then
%\[
%\lim_{n\to\infty}\frac{\log b_n}n\to\max\{c_1, \ldots, c_n\}.
%\]
%Let us write $k=\lfloor\alpha n\rfloor$.
For the first part, Theorem~\ref{T:sumsets} gives
\[
\lim_{n\to\infty}\frac{\log|H_n|}n
=\log 2
<\lim_{n\to\infty}\frac{\log|I_k|}n
=\alpha\log 4.
\]
Since $I_k\subseteq A_{n,k}$ and $|A_{n,k}|\le|H_n|+|I_k|$, it
follows that $\lim_{n\to\infty} \log|A_{n,k}|/n=\alpha \log 4.$
In other words, for the given range of parameters, the interval $I_k$
makes the dominant contribution to the size of $A_{n,k}$.

For the sumsets, we note that (again using Theorem~\ref{T:sumsets})
\begin{align*}
\lim_{n\to\infty}\frac{\log|H_n+H_n|}n &=\log 3\\
\intertext{and}
\lim_{n\to\infty}\frac{\log|H_n+I_k|}n
&=\frac{1+\alpha}2\log 4>\lim_{n\to\infty}\frac{\log|I_k+I_k|}n.\\
\end{align*}
Since $H_n+H_n$ and $H_n+I_k$ are both subsets of $A_{n,k}+A_{n,k}$,
and
\[
|A_{n,k}+A_{n,k}|\le|H_n+H_n|+|H_n+I_k|+|I_k+I_k|
\]
it follows that
\[
\lim_{n\to\infty}\frac{\log|A_{n,k}+A_{n,k}|}n=\max\left\{\log 3,\frac{1+\alpha}2\log 4\right\}.
\]
In other words, the dominant contribution to the size of $A_{n,k}+A_{n,k}$ comes from
$H_n+H_n$ if $\alpha<\log 3/\log 2-1\approx 0.585$, and from $H_n+I_k$ otherwise.

Finally, for the dilates, we have
\[
\lim_{n\to\infty}\frac{\log|A_{n,k}+2\cdot A_{n,k}|}n=\lim_{n\to\infty}\frac{\log|H_n+2\cdot H_n|}n=\log 4.
\]
\end{proof}

\noindent We can now use Corollary~\ref{C:technicallowerbound} to expand the known feasible region $F_{1,2}.$
Define $f:[1,2]\to[1,2]$ by
\[
f(x)=\begin{cases}
\frac12(\beta+1) &\text{if $1\le x\le\frac{\log 2}{\log (3/2)}$}\\
(\log_4 3)x &\text{if $\frac{\log 2}{\log (3/2)}\le x\le 2$}.
\end{cases}
\]

\begin{cor}\label{C:feas12}
For all $1<\beta<2$, $\left(f(\beta),\beta\right)\in F_{1,2}$.
\end{cor}
\begin{proof}
Let $\alpha=1/\beta$, and set $k=\lfloor\alpha n\rfloor$. Then
\begin{align*}
\alpha\log 4f(\beta)
&=\alpha\log 4\max\left\{\frac12(\beta+1),(\log_4 3)\beta\right\}\\
&=\alpha\log 4\max\left\{\frac{1+\alpha}{2\alpha},\frac{\log 3}{\alpha\log 4}\right\}\\
&=\max\set{\frac{1+\alpha}2\log 4,\log 3}.
\end{align*}
Thus Corollary~\ref{C:technicallowerbound} provides sets $A_{n,k}$ with
\begin{align*}
\frac{\log|A_{n,k}|}n&\to\alpha\log 4\\
\frac{\log|A_{n,k}+A_{n,k}|}n
&\to\alpha f(\beta)\log 4\\
\frac{\log|A_{n,k}+2\cdot A_{n,k}|}n
&\to \alpha\beta\log 4,
\end{align*}
proving the result.
\end{proof}

\noindent A simple change of variable shows that the graph of the piecewise-linear function
\[
y=\min\left(2x-1,(\log_34)x\right)=
\begin{cases}
2x-1&1\le x\le\log_{\frac94}3=1.3548\ldots\\
(\log_34)x&\log_{\frac94}3\le x\le 2\\
\end{cases}
\]
is entirely contained in $F_{1,2}$.

As mentioned above, this gives a partial converse to Theorem~\ref{thm:hp1}.

\begin{cor}
For all $\epsilon>0$, there exist sets $S$ and numbers $K>1$ with
$|S+S|\le K|S|$ but with $|S+2\cdot S|>K^{2-\epsilon}|S|.$
\end{cor}

\begin{proof}
Take $\alpha>\log3/\log 2-1$ in Corollary~\ref{C:technicallowerbound}.
\end{proof}

Let us quickly discuss a lower bound on the feasible region.

\begin{thm}\label{T:upperfromplunecke}
For all sets $A$, $|A||A+A|\le |A+2\cdot A|^2$.	
\end{thm}

\begin{proof}
Corollary 7.3.6 of~\cite{Zhao}, which is an easy conseqeunce of Pl\"unnecke's inequality,
states that for any three sets $A,B,C$,
\[
|A||B+C| \le |A+B||A+C|.
\]
Setting $B=C=2\cdot A$ gives
\[
|A||A+A| = |A||2\cdot A+2\cdot A| \le |A+2\cdot A|^2.
\]
\end{proof}

These results are all illustrated in Figure~\ref{fig:feas}. First, the two results
of Hanson and Petrridis (Theorems~\ref{thm:hp1} and~\ref{thm:hp2}) show that the regions
$y>2.95x-1.95$ and $y>4x/3$ are both {\em infeasible} (i.e., none of the points in those
regions is attainable). Likewise, Theorem~\ref{T:upperfromplunecke} shows that the region
$y<1+x/2$ is also infeasible. In the other direction, Corollary~\ref{C:feas12} shows that
the lines $OD$ and $DC$ are feasible, while Theorem~\ref{T:basicbutnecessary} shows that
the line $OE$ is feasible. Consequently, the entire quadrilateral $ODCE$ is feasible.

\begin{figure}[h]
\begin{center}
	\includegraphics[scale=.3]{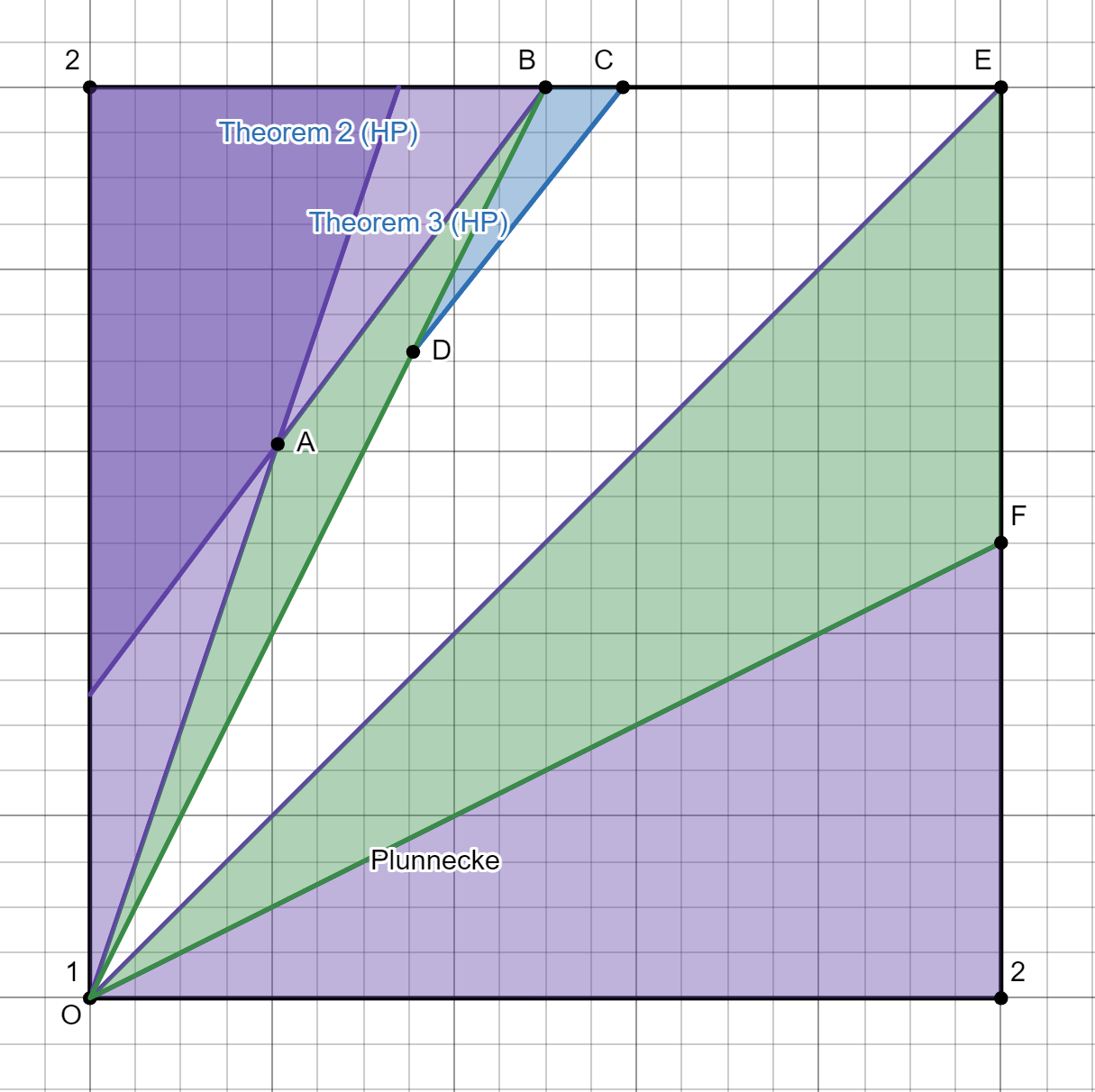}
\end{center}
	\caption{The feasible region $F_{1,2}$}\label{fig:feas}
\end{figure}

These results leave three regions unexplored: triangles $OAB, BCD$ and $OEF$. More
precisely, we can ask three questions:

\begin{qn}\label{Q:2 is king}
Is (the interior of) triangle $OAB$ infeasible? In other
words, does $|A+A|=|A|^{1+t}$ imply $|A+2\cdot A|\le|A|^{1+2t}$? Alternatively,
is it true that, for all $A$, $|A||A+2\cdot A|\le|A+A||A+A|$?
\end{qn}

\begin{qn}\label{Q:log 4 over log 3}
Are there no attainable points above the extension of line $CD$? In other words, is it
true that, if $|A+A|=3^t$, then $|A+2\cdot A|\le 4^t$?
\end{qn}

\begin{qn}\label{Q:mst2d}
Is (the interior of) triangle $OEF$ infeasible? In other words,
is it true that, for all $A$, $|A+2\cdot A|\ge|A+A|$?	
\end{qn}

We coin the term {\em MST2D sets} (or {\em more sums than 2-dilates sets}) for
counterexamples to Question~\ref{Q:mst2d}.  A natural candidate for an MST2D set is
a Sidon set (that is, a set for which $|A+A|$ is as large as possible).
However, one can easily show that a Sidon set cannot be an MST2D set.

\begin{lem}
If $A$ is a Sidon set with at least two elements, then $|A+2\cdot A|>|A+A|$.	
\end{lem}

\begin{proof}
Adding and multiplying non-zero constants to $A$ does not change $|A+A|$ or $|A+2\cdot A|$.
Thus we can assume that
$0\in A$ and $\gcd(A)=1$.

Let $n=|A|$. If $X_1, X_2$ are independent and identically distributed (IID) random variables,
drawn from any distribution on a finite set $S$, then $\Prb(X_1=X_2)\ge 1/|S|$. Thus if
$A_1, A_2, A_3, A_4$ are IID drawn from the uniform (or indeed any) distribution on $A$,
then
\[
|A+2\cdot A|\ge \Prb(A_1+2\cdot A_2=A_3+2\cdot A_4)^{-1}.
\]
Now, since $A$ is a Sidon set,
\[
\Prb(A_4-A_2=k)=\Prb(A_1-A_3=k)=\begin{cases}
1/n, \textrm{ if }k=0\\
0, \textrm{ if }k\notin A-A\\
1/n^2, \textrm{ otherwise}.
\end{cases}
\]
Thus
\[
\Prb(A_1+2\cdot A_2=A_3+2\cdot A_4)=\Prb(A_1-A_3=2(A_4-A_2))=\frac{1}{n^2}+\frac{K}{n^4},
\]
where $K$ is the number of
non-zero elements of $A-A$ which are double some other element of $A-A$.
Now there are $n^2-n$ non-zero elements of $A-A$, but $A$ contains elements of both parities
(since $0\in A$ and $\gcd(A)=1$). Therefore, $A-A$ contains at least $2(n-1)$ odd elements,
so $K\le n^2-3n+2$.
Consequently,
\begin{align*}
|A+2\cdot A|&\ge\left(\frac{1}{n^2}+\frac{n^2-3n+2}{n^4}\right)^{-1}\\
%&=n^2\left(2-\frac{3}{n}+\frac{2}{n^2}\right)^{-1}\\
&=\frac{n^2}{2}\left(1-\frac{3}{2n}+\frac{1}{n^2}\right)^{-1}\\
&\ge\frac{n^2}{2}\left(1+\frac{3}{2n}-\frac{1}{n^2}\right)\\
&=\frac{n^2}{2}+\frac{3n}{4}-\frac12\ge\frac{n^2+n}{2}=|A+A|.
\end{align*}
\end{proof}

Similarly, a natural candidate for a counterexample to Question~\ref{Q:log 4 over log 3} is
a ``2-Sidon" set (i.e., one where $|A+2\cdot A|$ is as large as possible). An easy way to construct such sets is
as subsets of the hypercube $\{0, 1\}^n$. But it follows from Theorem 2.3 of~\cite{Green} that such sets in fact satisfy the
condition of Question~\ref{Q:log 4 over log 3}.

\begin{lem}\label{L:Woodall}[Green \cite{Green}]
Suppose $A$ and $B$ are subsets of the hypercube $\{0, 1\}^n$. Then
\[
|A+2\cdot B|=|A||B|\le|A+B|^p,
\]
where $p=\log 4/\log 3$.	
\end{lem}

\noindent This lemma has an interesting history, going back to the 1970s. Details are in Appendix B of~\cite{Green}.

Very recently, Becker, Ivanisvili, Krachun and Madrid~\cite{Becker} proved that subsets of $\{0, 1\}^n$ also satisfy
the conditions of Question~\ref{Q:2 is king}; however, since in their result $|A+2\cdot A|=|A|^2$, their Corollary 2
is also a corollary of Lemma~\ref{L:Woodall} above.

\begin{lem}[Becker et al. \cite{Becker}, Corollary 2]
Suppose $A$ is a subset of the hypercube $\{0, 1\}^n$. Then
\[
\frac{|A+2\cdot A|}{|A|}\le \left(\frac{|A+A|}{|A|}\right)^q,
\]	
where $q=\log 2/\log(3/2)$.
\end{lem}

\section{Fractional Dilates}\label{S:Dilates 1}

\subsection{General results on norms}

We remind the reader of the following definitions. A {\em fractional dilate $\gamma$}
is a map $\gamma:\Z\to\R^{+}\cup\{0\}$ with finite support ${\rm supp}(\gamma)$.
We define the {\em size of a fractional dilate} to be
\[
\norm{\gamma}=\inf_{0\le p\le1}\sum_{n\in{\rm supp}(\gamma)}\gamma(n)^p.
\]
A {\em fractional set} is a fractional dilate $\alpha$ for which
$\alpha(n)\le 1$ for all $n\in \Z$. Finally, we describe
a fractional dilate as being {\em opulent}, {\em spartan} or {\em p-comfortable}
if the above infimum is attained at $p=0$, $p=1$ or $0<p<1$ respectively, so that,
for instance, all fractional sets are spartan.

First, we give a simple characterisation of fractional dilates, which enables us
to easily decide whether a fractional dilate is opulent, spartan or $p$-comfortable.

\begin{thm}\label{T:fractionalsize}
A fractional dilate $\gamma$ with support $S={\rm supp}(\gamma)$ is
\[\begin{cases}
\text{spartan,} & \text{if $\sum_{n\in S}\gamma(n)\log\gamma(n)\le 0$}\\
\text{opulent,} & \text{if $\sum_{n\in S}\log\gamma(n)\ge 0$}\\
\text{$p$-comfortable,} & \text{if $\sum_{n\in S}\gamma(n)^p\log\gamma(n)=0.$}
\end{cases}\]
\end{thm}
\begin{proof}
For a fixed $\gamma$ with support $S$, define a function $f:[0,1]\to\R$ by
\[
f(p)=\sum_{n\in S}\gamma(n)^p,
\]
so that $\norm{\gamma}=\inf_{0\le p\le 1}f(p)$. $f$ is twice-differentiable, and also strictly convex, since
\[
f''(p)=\sum_{n\in S}(\log\gamma(n))^2\gamma(n)^p>0.
\]

\medskip

\noindent Suppose first that
\[
f'(1)=\sum_{n\in S}\gamma(n)\log\gamma(n)\le 0.
\]
Then we must have $f'(p)\le 0$ for all $0<p<1$, and hence $\norm{\gamma}=f(1)$, i.e., $\gamma$ is spartan.

\medskip

\noindent Suppose next that
\[
f'(0)=\sum_{n\in S}\log\gamma(n)\ge 0.
\]
Then we must have $f'(p)\ge 0$ for all $0<p<1$, and hence $\norm{\gamma}=f(0)$, i.e., $\gamma$ is opulent.

\medskip

\noindent Otherwise, $f'(0)<0<f'(1)$, and hence there is a unique $p\in(0,1)$ for which
\[
f'(p)=\sum_{n\in S}\gamma(n)^p\log\gamma(n)=0.
\]
It follows that $\norm{\gamma}=f(p)$, i.e., $\gamma$ is $p$-comfortable.
\end{proof}

Next we give yet another characterisation of $\norm{\gamma}$. Recall that, for positive numbers $y_1, \ldots, y_n$
summing to 1, the entropy function $H(y_1, \ldots, y_n)$
is defined to be
\[
H(y_1, \ldots, y_n)=-\sum_{i=1}^n y_i\log_2 y_i.
\]
Gibbs' inequality (see for instance~\cite{GibbsInequality}) states that
\[
H(y_1, \ldots, y_n)\le-\sum y_i\log_2z_i
\]
for any sequence of positive $z_i$ summing to $1$, with equality if and
only if $y_i=z_i$ for all $i$.

\begin{lem}\label{C:Entropy And Scale of Norm}
Suppose $\gamma$ is a fractional dilate with support
$S=\{s_1, \ldots, s_n\}$. Then
\[\norm{\gamma}=\max_{y_1+\cdots+y_n=1}2^{H(y_1,\ldots,y_n)}
\min\set{1,\prod_{i=1}^n\gamma(s_i)^{y_i}}.
\]
\end{lem}

\begin{proof}
First we observe that, for real $x$ and $0\le p\le 1$, we have
$\min\set{0, x}\le px$, with equality exactly when either
\begin{enumerate}
\item $p=0$ and $x\ge 0$,
\item $p=1$ and $x\le 0$, or
\item $0<p<1$ and $x=0$.
\end{enumerate}

\noindent Next, fix $0\le p\le 1$, and let
\[
z_i=\frac{\gamma(s_i)^p}{\sum_{i}\gamma(s_i)^p}.
\]
Clearly $\sum z_i=1$. Then, for all positive $y_i$ summing to 1,
we have
\begin{align*}
H(y_1,\ldots,y_n)+\min\set{0,\sum_iy_i\log_2\gamma(s_i)}
&\le H(y_1,\ldots,y_n)+\sum_i py_i\log_2\gamma(s_i)\\
&=H(y_1,\ldots,y_n)+\sum_i y_i\log_2\gamma(s_i)^p\\
&=H(y_1,\ldots,y_n)+\sum_i y_i\left(\log_2z_i+\log_2\sum_{i}\gamma(s_i)^p\right)\\
&=H(y_1,\ldots,y_n)+\sum_i y_i\log_2z_i+\log_2\sum_{i}\gamma(s_i)^p\\
&\le\log_2\sum_{i}\gamma(s_i)^p,
\end{align*}
with the last inequality being Gibbs' inequality.

Raising 2 to both sides shows that, for all positive $y_1,\ldots,y_n$ summing to 1, and all $0\le p\le 1$,
\begin{equation}\label{minimax}
2^{H(y_1,\ldots,y_n)}
\min\set{1,\prod_i \gamma(s_i)^{y_i}}\le\sum_{i}\gamma(s_i)^p.
\end{equation}
To prove the theorem, we only need show that we can choose the $y_i$ and $p$ to achieve
equality in~\eqref{minimax}. For this, revisiting the derivation of~\eqref{minimax}, we require
that $y_i=z_i$ for all $i$ and also that
\[
\min\set{0,\sum_iy_i\log_2\gamma(s_i)}=\sum_i py_i\log_2\gamma(s_i).
\]
From the definition of $z_i$, this means we require exactly that
\[
p\sum_i\gamma(s_i)^p\log\gamma(s_i)=\min\set{0,\sum_i\gamma(s_i)^p\log\gamma(s_i)}.
\]
As discussed at the start of this proof, this holds exactly when
\begin{enumerate}
\item $p=0$ and $\sum_i\log\gamma(s_i)\ge 0$, i.e., when $\gamma$ is opulent,
\item $p=1$ and $\sum_i\gamma(s_i)\log\gamma(s_i)\le 0$, i.e., when $\gamma$ is spartan, or
\item $0<p<1$ and $\sum_i\gamma(s_i)^p\log\gamma(s_i)=0$, i.e., when $\gamma$ is $p$-comfortable;
\end{enumerate}
here, we have also used Theorem~\ref{T:fractionalsize}.
Consequently, we can indeed achieve equality in~\eqref{minimax}, so the theorem is proved.
\end{proof}	

\subsection{Results for two sets}

We recall some more definitions from the introduction. Given a fractional set $\alpha$,
we say that a random set $S_n\subseteq\Z^{n}$ is drawn from $\alpha^n$ if each element of
$\Z^{n}$ is chosen independently, and the probability that $(i_1, i_2, \ldots, i_n)$ is selected is
$\alpha(i_1)\alpha(i_2)\ldots\alpha(i_n)$. Moreover, for fractional sets $\alpha, \beta$ and an
integer $k$, let $\alpha+k\cdot\beta$ denote the fractional dilate defined by the formula
\[
(\alpha+k\cdot\beta)(n)=\sum_{\substack{(i,j)\\i+kj=n}}\alpha(i)\beta(j).
\]

\noindent First we prove two simple lemmas that we will use repeatedly. In the proof of both
we use the fact that fractional sets are spartan.

\begin{lem}\label{L:Product}
Suppose $\alpha$ and $\beta$ are fractional sets, and $\gamma=\alpha+\beta$ is spartan.
Then
\[
\norm{\gamma}=\norm{\alpha}\norm{\beta}.
\]
\end{lem}
\begin{proof}
If $\gamma$ is spartan with support $S$, then
\[
\norm{\gamma}=\sum_{n\in S}\gamma(n)=\sum_{n\in S}\sum_{\substack{(i,j)\\i+j=n}}\alpha(i)\beta(j)
=\sum_{i\in\Z}\alpha(i)\sum_{j\in\Z}\beta(j)=\norm{\alpha}\norm{\beta}.
\]
\end{proof}

\begin{lem}\label{L:expectation}
Let $\alpha$ be a fractional set, and suppose that $S_n\subseteq\Z^{n}$ is drawn from $\alpha^n$.
Then
\[
E|S_n|=\norm{\alpha}^n.
\]
\end{lem}
\begin{proof}
Writing $v=(v_1,\ldots,v_n)\in\Z^n$, we have
\[
\E|S_n|=\sum_{v\in\Z^n}\Prb(v\in S_n)=\sum_{v\in\Z^n}\alpha(v_1)\cdots\alpha(v_n)=\prod_{i=1}^n\sum_{v_i\in Z}\alpha(v_i)=\norm{\alpha}^n.
\]
\end{proof}

Our aim is to prove Theorems~\ref{T:fractional} and~\ref{T:spartan}, which concern just one fractional set $\alpha$
and a single random set $S_n\subseteq \Z^n$ drawn from $\alpha^n$. However, it is easier to start with {\it two}
fractional sets $\alpha$ and $\beta$, and let $S_n, T_n\subseteq\Z^n$ be drawn independently from $\alpha^n$
and $\beta^n$ respectively. In this section, we consider two such sets, and prove the following ``two-set" versions
of Theorems~\ref{T:fractional} and~\ref{T:spartan}.

\begin{thm}\label{T:fractional two}
Let $\alpha$ and $\beta$ be fractional sets, and suppose
$S_n, T_n\subseteq\Z^{n}$ are drawn from $\alpha^n$ and $\beta^n$ respectively.
Then
\[
\lim_{n\to\infty}(\E|S_n+k\cdot T_n|)^{1/n}=\norm{\alpha+k\cdot\beta}.
\]
\end{thm}
\begin{thm}\label{T:spartan_two}
Let $\alpha$ and $\beta$ be fractional sets, and suppose
$S_n, T_n\subseteq\Z^{n}$ are drawn from $\alpha^n$ and $\beta^n$ respectively.
If $\gamma=\alpha+k\cdot\beta$ is strictly spartan, in the sense that
\[
\sum_{n\in{\rm supp}(\gamma)}\gamma(n)\log\gamma(n)<0,
\]
then with high probability
\[
|S_n+k\cdot T_n|\ge\tfrac12|S_n||T_n|.
\]
\end{thm}

Let us first prove the upper bound in Theorem~\ref{T:fractional two}. For this, we require
a definition. For two sets $A,B\subseteq \Z^n$, the {\em multiplicity} $\mult_{A+k\cdot B}(x)$ of
$x$ in $A+k\cdot B$ is defined by the formula
\[
\mult_{A+k\cdot B}(x)=|A\cap(x-k\cdot B)|=|\{(a,b):a\in A,b\in B,a+kb=x\}|.
\]
In other words, $\mult_{A+k\cdot B}(x)$ is the number of ways of writing $x=a+kb$ with $a\in A$ and $b\in B$.

By replacing $k\cdot\beta$ by $\beta$ in Theorem~\ref{T:fractional two}, it is enough to prove the theorem for $k=1$.
In the rest of this subsection, we will make this simplification.

\begin{thm}\label{T:fractional2setupper}
Let $\alpha$ and $\beta$ be fractional sets, suppose
$S_n, T_n\subseteq\Z^{n}$ are drawn from $\alpha^n$ and $\beta^n$ respectively,
and let $\gamma=\alpha+\beta$.
Then
\[
\E|S_n+T_n|\le\norm{\gamma}^n.
\]
\end{thm}

\begin{proof}
Let $X$ be the support of $\gamma$. Then the possible elements of $S_n+T_n$
are the elements of $X^n$. Given a particular element $x\in X^n$, we have,
for all $0\le p\le 1$,
\begin{align*}
\Prb(x\in S_n+T_n)&=\Prb(\mult_{S_n+T_n}(x)>0)\\
&\le\min\{1,\E(\mult_{S_n+T_n}(x))\}\\
&\le(\E(\mult_{S_n+T_n}(x)))^p.
\end{align*}

\noindent Now for $x=(x_1, \ldots, x_n)\in X^n$, we have
\begin{align*}
\E(\mult_{S_n+T_n}(x))
&=\sum_{\substack{z_1+y_1=x_1\\ \vdots\\ z_n+y_n=x_n}}
\alpha(z_1)\cdots\alpha(z_n)\beta(y_1)\cdots\beta(y_n)
\\
&=\left(\sum_{z_1+y_1=x_1}\alpha(z_1)\beta(y_1)\right)\cdots
\left(\sum_{z_n+y_n=x_n}\alpha(z_n)\beta(y_n)\right)\\
&=\gamma(x_1)\gamma(x_2)\cdots\gamma(x_n).
\end{align*}

\noindent It follows that
\begin{align*}
\E|S_n+T_n|
&=\sum_{x\in X^n}\Prb(x\in S_n+T_n)\\
&\le\sum_{x\in X^n}(\E(\mult_{S_n+T_n}(x)))^p\\
&=\sum_{(x_1,\ldots,x_n)\in X^n}\gamma(x_1)^p\gamma(x_2)^p\cdots\gamma(x_n)^p\\
&=\left(\sum_{x\in X}\gamma(x)^p\right)^n
\end{align*}
for all $0\le p\le 1$.  Since $\norm\gamma=\inf_{0\le p\le1}\sum_x\gamma(x)^p$, the result follows.
\end{proof}

Next we prove that $\norm{\gamma}$ is a lower bound for the limit. We will require a series of lemmas.

\begin{lem}\label{L:tiny}
Suppose that $X=\sum_{i=1}^NZ_i$, where the $Z_i$ are independent Bernoulli random variables.
Then
\[
\Prb(X>0)\ge \E(X)-\tfrac12\E(X)^2.
\]
\end{lem}
\begin{proof}
We have
\[
\E(X)^2=\sum_{i=1}^N\sum_{j=1}^N\E(Z_i)\E(Z_j)=\sum_{i=1}^N\sum_{j=1}^N\E(Z_iZ_j)\\
\ge 2\sum_{i<j}\E(Z_iZ_j)\\
=2\E\binom{X}{2},
\]
so that, since $n-\binom{n}{2}\le\mathbbm{1}_{n>0}$ for all $n\ge 0$,
\[
\E(X)-\tfrac12\E(X)^2\le\E\left(X-\binom{X}{2}\right)\le\E(\mathbbm{1}_{X>0})=\Prb(X>0).
\]
\end{proof}

\begin{lem}\label{L:technical}
Suppose that $(X_n)_{n=1}^{\infty}$ is a collection of random variables,
each of which can be written as the sum of a finite number of independent Bernoulli random variables.
If
\[
\lim_{n\to\infty}\E(X_n)^{1/n}=t
\]
then
\[
\lim_{n\to\infty}\Prb(X_n>0)^{1/n}=\min(1,t).
\]
\end{lem}

\begin{proof}
Suppose first that $t<1$. Lemma~\ref{L:tiny} implies that
\[
\limsup_{n\to\infty}\left(\E(X_n)-\Prb(X_n>0)\right)^{1/n}\le
\lim_{n\to\infty}\left(\tfrac12\E(X_n)^2\right)^{1/n}=t^2.
\]
Consequently, for all $\epsilon>0$, if $n\ge n_0(\epsilon)$, we have both
\[
(t-\epsilon)^n\le\E(X_n)\le(t+\epsilon)^n
\]
and
\[
\E(X_n)-\Prb(X_n>0)\le (t^2+\epsilon)^n
\]
so that also
\[
(t-2\epsilon)^n\le\Prb(X_n>0)\le(t+\epsilon)^n
\]
which proves that $\Prb(X_n>0)^{1/n}\to t$.

Next suppose that $t\ge 1$. Given any $0<u<1$, write $Y_n=X_nW_n$,
where the $W_i$ are new independent Bernoulli random variables with
$\Prb(W_n=1)=(u/t)^n$. Then
\[
\lim_{n\to\infty}\E(Y_n)^{1/n}=\lim_{n\to\infty}(\E(X_n)\E(W_n))^{1/n}=u,
\]
and so by the above argument
\[
u=\lim_{n\to\infty}\Prb(Y_n>0)^{1/n}\le\liminf_{n\to\infty}\Prb(X_n>0)^{1/n}\le 1.
\]
Since this is true for all $u<1$, we must have $\Prb(X_n>0)^{1/n}\to 1$.
\end{proof}

We can use Lemma~\ref{L:technical} to calculate the asymptotic behaviour of the
probability that a randomly chosen vector lies in $S_n+T_n$.

\begin{cor}\label{C:abittechnical}
Let $\alpha$ and $\beta$ be fractional sets, suppose
$S_n, T_n\subseteq\Z^{n}$ are drawn from $\alpha^n$ and $\beta^n$ respectively,
and let $\gamma=\alpha+\beta$. Fix $N>0$, and suppose that

\medskip

$\bullet$ $x_1, \ldots, x_N\in\Z$

\smallskip

$\bullet$ $y_1, \ldots, y_N\ge 0$ with $\sum y_i=1$

\smallskip

$\bullet$ for each $n$, $z_{1,n}, \ldots, z_{N,n}\in\Z_{\ge0}$ with $\sum_i z_{i,n}=n$ and $z_{i,n}/n\to y_i$ for each $i$

\smallskip

$\bullet$ for each $n$, $v_n\in\Z^n$ is such that $z_{i,n}$ coordinates of $v_n$ are equal to $x_i$.

\medskip

\noindent Then, if
\[
t=\gamma(x_1)^{y_1}\ldots\gamma(x_N)^{y_N}
\]
we have
\[
\lim_{n\to\infty}\Prb(v_n\in S_n+T_n)^{1/n}
=\min\set{1,t}.\]
\end{cor}
\begin{proof}
For a fixed sequence $v_n$, let
\[
X_n=\mult_{S_n+T_n}(v_n)=\sum_{z\in\Z^n}\mathbbm{1}_{z\in S_n,v_n-z\in T_n},
\]
so that each $X_n$ is a sum of a finite number of independent Bernoulli random variables.
As in the proof of Theorem~\ref{T:fractional2setupper},
\[
\E(X_n)=\gamma(x_1)^{z_{1,n}}\gamma(x_2)^{z_{2,n}}\ldots\gamma(x_N)^{z_{N,n}},
\]
so $\E(X_n)^{1/n}\to t$. Applying Lemma~\ref{L:technical}, we get that
\[
\Prb(v_n\in S_n+T_n)^{1/n}=\Prb(X_n>0)\to\min\set{1,t}.
\]
\end{proof}

\noindent The next corollary proves the lower bound on $(\E|S_n+T_n|)^{1/n}$
which, together with Theorem~\ref{T:fractional2setupper}, completes the proof of Theorem~\ref{T:fractional two}.

\begin{cor}\label{C:full result two sets finite support}
Let $\alpha$ and $\beta$ be fractional sets,
%with finite support
suppose
$S_n, T_n\subseteq\Z^{n}$ are drawn from $\alpha^n$ and $\beta^n$ respectively,
and let $\gamma=\alpha+\beta$.
Then
\[
\liminf_{n\to\infty}(\E|S_n+T_n|)^{1/n}\ge\norm{\gamma}.
\]	
\end{cor}

\begin{proof}
Since $\alpha$ and $\beta$ have finite support, so does $\gamma$. Let
$S={\rm supp}(\gamma)=\{s_1, \ldots, s_N\}$.
By Lemma~\ref{C:Entropy And Scale of Norm}, there exist non-negative numbers $y_1, \ldots, y_N$
summing to 1 with
\[
\norm{\gamma}=
2^{H(y_1,\ldots,y_N)}\min\set{1,\prod_{i=1}^N \gamma(s_i)^{y_i}}.
\]

For each $n$, choose integers $z_{1,n}\ldots,z_{N,n}$ summing to $n$, and with
$z_i/n\to y_i$ for each $1\le i\le N$. Let $V_n\in\Z^n$ be the set of vectors
with exactly $z_{i,n}$ coordinates equal to $s_i$, for each $1\le i\le N$. Then,
for each $v\in V_n$, Corollary~\ref{C:abittechnical} shows that
\[
\lim_{n\to\infty}\Prb(v\in S_n+T_n)^{1/n}=\min\set{1,\prod_{i=1}^N\gamma(s_i)^{y_i}}.
\]

\noindent It is well known that, abbreviating $z_{i,n}$ to $z_i$,
\[
\lim_{n\to\infty}|V_n|^{1/n}=\lim_{n\to\infty}\binom{n}{z_1,z_2,\ldots,z_N}^{1/n}=2^{H(y_1,\ldots,y_N)}.
\]

\noindent Consequently,
\begin{align*}
\liminf_{n\to\infty}(\E|S_n+T_n|)^{1/n}
&\ge\liminf_{n\to\infty}(\E|(S_n+T_n)\cap V_n|)^{1/n}\\
&=\liminf_{n\to\infty}(|V_n|\cdot\Prb(v\in S_n+T_n|v\in V_n))^{1/n}\\
&=\lim_{n\to\infty}|V_n|^{1/n}\cdot\lim_{n\to\infty}\Prb(v\in S_n+T_n|v\in V_n)^{1/n}\\
%&=\lim_{n\to\infty}\left\{\binom{n}{z_1, z_2, \ldots, z_N}\Prb(v\in S_n+T_n|v\in V_n)\right\}^{1/n}\\
&=2^{H(y_1, y_2,\ldots, y_N)}\min\set{1, \prod_{i=1}^N \gamma(s_i)^{y_i}}\\
&=\norm{\gamma}.
\end{align*}
\end{proof}

\noindent With Theorem~\ref{T:fractional two} proved, we turn to Theorem~\ref{T:spartan_two}.

\begin{lem}\label{L:sparta for two sets}
Let $\alpha$ and $\beta$ be fractional sets, and
%with finite support,
suppose
$S_n, T_n\subseteq\Z^{n}$ are drawn from $\alpha^n$ and $\beta^n$ respectively,
and let $\gamma=\alpha+\beta$. Suppose that $\gamma$ is strictly spartan, in the sense that
\[
\sum_{n\in{\rm supp}(\gamma)}\gamma(n)\log\gamma(n)<0.
\]
Then
\[
\E(|S_n||T_n|)=\E|S_n|\cdot\E|T_n|=\norm{\alpha}^n\norm{\beta}^n=\norm{\gamma}^n
\]
and
\[
\E(|S_n||T_n|-|S_n+T_n|)=o(\norm{\gamma}^n).
\]
\end{lem}

\begin{proof}
The first part of the conclusion follows from Lemmas~\ref{L:Product} and~\ref{L:expectation}.
For the second part, for $v=(v_1,\ldots,v_n)\in\Z^n$, write
\[
X_v=\E(\mult_{S_n+T_n}(v)).
\]
Since $X_v$ is a sum of independent Bernoulli random variables, Lemma~\ref{L:tiny} shows that
\[
\E(X_v)-\Prb(X_v>0)\le \E(X_v)^2.
\]
Since the left hand side is also at most $\E(X_v)$, it follows that
\[
\E(X_v)-\Prb(X_v>0)\le \E(X_v)^p\textrm{ for all }1\le p\le 2.
\]
Therefore, for all $1\le p\le 2$,
\begin{align*}
\E(|S_n||T_n|-|S_n+T_n|)
&=\sum_v(\E(X_v)-\Prb(X_v>0))\leq \sum_v \E(X_v)^p\\
&=\sum_{v_1, v_2, \ldots, v_n}\gamma(v_1)^p\ldots\gamma(v_n)^p=\left(\sum_{v}\gamma(v)^p\right)^{n}.
\end{align*}
Now $\gamma$ is strictly spartan, so the function $f(p)=\sum_v\gamma(v)^p$ is strictly decreasing on the interval $[0,1+\epsilon]$,
for some $\epsilon>0$. Consequently, for that $\epsilon$, we have $\sum_v\gamma(v)^{1+\epsilon}<\norm\gamma$. Thus
$\E(|S_n||T_n|-|S_n+T_n|)=o(\norm{\gamma}^n)$ as $n\to\infty$.
\end{proof}

\noindent Theorem~\ref{T:spartan_two} follows from Lemma~\ref{L:sparta for two sets} and Markov's inequality.

\subsection{From two sets to one: the rainbow connection}\label{S:Dilates 2}

In this section we will prove Theorems~\ref{T:fractional}
and~\ref{T:spartan}.
Let $\alpha$ be a fractional set, let $k$ be a nonzero integer, and
let $\gamma:=\alpha+k\cdot\alpha$ denote the fractional dilate defined by
$\gamma(n)=\sum_{i+kj=n}\alpha(i)\alpha(j)$.

First we prove two easy lemmas.

\begin{lem}\label{L:variance}
Let $\alpha$ be a fractional set, and suppose that $S_n\subseteq\Z^{n}$ is drawn from $\alpha^n$.
Then
\[
\Var|S_n|\le\norm{\alpha}^n.
\]
\end{lem}
\begin{proof}
$|S_n|$ is the sum of independent Bernoulli random variables $X_i$ with $\Prb(X_i=1)=:p_i$. We have
\[
\Var|S_n|=\sum_{i\in S}\Var X_i
=\sum_{i\in S}p_i(1-p_i)\le\sum_{i\in S}p_i=\E|S_n|=\norm{\alpha}^n,
\]
where the last equality is Lemma~\ref{L:expectation}.
\end{proof}

\begin{lem}\label{L:we do use this}
Let $\alpha$ be a fractional set with $\norm\alpha>1$, and let $\gamma:=\alpha+k\cdot\alpha$
for a fixed nonzero integer $k$. Then $\norm\alpha\le\norm\gamma$.
\end{lem}

\begin{proof}
Let $S_n$ be drawn from $\alpha^n$.
Then $|S_n|$ can be written as the sum of independent Bernoulli random variables,
and so $(\E|S_n|)^{1/n}=\norm\alpha>1$ for all $n$ by Lemma~\ref{L:expectation}.
Thus, by Lemma~\ref{L:technical}, we have
\[
\lim_{n\to\infty}\Pr(|S_n|>0)^{1/n}=1.
\]
Now let $T_n$ be drawn independently from $\alpha^n$.
By Corollary~\ref{C:full result two sets finite support}, we have
\[
\norm{\gamma}=\lim_{n\to\infty}(\E|S_n+k\cdot T_n|)^{1/n}.
\]
But $|S_n+k\cdot T_n|\ge|T_n|$ whenever $S_n$ is non-empty, so
\[
|S_n+k\cdot T_n|\ge|T_n|\cdot\mathbbm{1}_{|S_n|>0}.
\]
Since $|S_n|$ and $|T_n|$ are independent, it follows that
\[
\norm{\gamma}=\lim_{n\to\infty}(\E|S_n+k\cdot T_n|)^{1/n}\ge
\lim_{n\to\infty}(\E|T_n|\cdot\Prb(|S_n|>0))^{1/n}=\norm\alpha.
\]
\end{proof}

We will prove Theorem~\ref{T:fractional}
by comparing the sizes of $S_n+k\cdot S_n$ and $S_n+k\cdot T_n$, where $S_n$ and $T_n$ are drawn
independently from $\alpha^n$. For $\alpha$ and $\gamma$ fixed, with supports $A$ and $\Gamma$ respectively,
we say that a vector $v\in\Z^n$ is {\em rainbow} if its components include at least one copy of each
element of $\Gamma$. Finally, let $R_n$ denote the set of rainbow vectors in $\Z^n$.

The cases $k\ne 1$ and $k=1$ require separate analyses. We treat the case $k\ne 1$ first.

\begin{thm}\label{T:rainbow connection} Fix a fractional set $\alpha$ with $\norm\alpha>1$ and a positive integer $n$,
and let $S_n$ and $T_n$ be drawn independently from $\alpha^n$.
If $k\ne 1$, and $v$ is a rainbow vector for $\gamma=\alpha+k\cdot\alpha$, then
\[
\Prb(v\in S_n+k\cdot S_n)=\Prb(v\in S_n+k\cdot T_n)
\]
and so
\[
\E|(S_n+k\cdot S_n)\cap R_n|=\E|(S_n+k\cdot T_n)\cap R_n|.
\]
\end{thm}

\begin{proof} Recall that $A={\rm supp}(\alpha)$. For a fixed rainbow vector $v\in\Z^n$, write
\[
S=\{x\in\Z^n: x\in A^n, v-kx\in A^n\}=\{x\in\Z^n: \Prb(x\in S_n,v-kx\in S_n)>0\}.
\]
Since $A$ is finite, so is $S$.

We claim that there exist distinct $a,b\in A$ such that the expression
$a+kb$ is unique, i.e., $a+kb$ cannot be expressed in any other way $a'+kb'$, where $a',b'\in A$. Indeed, if $k<0$,
let $a=\max A$ and $b=\min A$. If $a',b'\in A$, then $a\ge a'$ and $b\le b'$, so
$a+kb\ge a'+kb'$, with equality only if $a=a'$ and $b=b'$. Similarly, if $k>1$,
we can take $b=\max A$ and $a=\max (A\setminus\{b\})$. If $a',b'\in A$ with $a'+kb'=a+kb$ but with $a'\ne a$ and $b'\ne b$,
we must have $b'<b$, whence $b'\le a$ and also $a'\le b$,
so that
\[
a'+kb'\le b+ka<b+ka+(k-1)(b-a)=a+kb,
\]
a contradiction.

Now, since $v$ is a rainbow vector, there exists $i$ with $v_i=a+kb$. It follows that for all $x\in S$, $x_i=a$ and $(v-kx)_i=b\ne a$,
so that $v-kx\notin S$ and $x\ne v-kx$. Consequently, the events $x\in S_n$ and $v-kx\in S_n$ are independent, and hence
\[
\Prb(x\in S_n,v-kx\in S_n)=\Prb(x\in S_n)\Prb(v-kx\in S_n).
\]
Furthermore, since the sets $\{x, v-kx\}$ for all $x\in S$ are disjoint, we have by independence
\begin{align*}
\Prb(v\in S_n+S_n)
&=1-\prod_{x\in S}(1-\Prb(x\in S_n,v-kx\in S_n))\\
&=1-\prod_{x\in S}(1-\Prb(x\in S_n)\Prb(v-kx\in S_n))\\
&=1-\prod_{x\in S}(1-\Prb(x\in S_n)\Prb(v-kx\in T_n))\\
&=1-\prod_{x\in S}(1-\Prb(x\in S_n,v-kx\in T_n))\\
&=\Prb(v\in S_n+T_n).
\end{align*}
The statement about expectations follows by summing over all $v\in R_n$.
\end{proof}

For $k=1$, $a+kb=b+ka$, so $S_n+S_n$ will be usually be a lot smaller than $S_n+T_n$. Thus we need to modify our strategy.
Order the elements of $\Z^n$ lexicographically, so that, for $v,v'\in\Z^n$, if $v$ and $v'$ first differ in the $i$th
coordinate, we set $v<v'$ if $v_i<v'_i$. Then, for subsets $U,V\subseteq\Z^n$, we define
\[
U+^<V=\setof{u+v}{u\in U, v\in V, u<v}.
\]

\begin{thm}\label{T:rainbow connection II}
Fix a fractional set $\alpha$ with $\norm\alpha>1$ and a positive integer $n$,
and let $S_n$ and $T_n$ be drawn independently from $\alpha^n$.
If $v$ is a rainbow vector for $\gamma=\alpha+\alpha$, then
\[
\Prb(v\in S_n+S_n)=\Prb(v\in S_n+^<T_n)
\]
and so
\[
\E|(S_n+S_n)\cap R_n|=\E|(S_n+^<T_n)\cap R_n|.
\]
\end{thm}

\begin{proof}
First we argue that, since $v$ is rainbow, $v/2\not\in A^n=({\rm supp}(\alpha))^n$. For let
$a$ and $b$ be the two largest elements of $A$. Then $c=a+b\in\Gamma={\rm supp}(\gamma)$,
so, since $v$ is a rainbow vector, there exists $i$ such that $v_i=c$. But, by choice of
$a$ and $b$, $(v/2)_i=c/2=(a+b)/2\not\in A$, so $v/2\not\in A^n$.

Next, as before, we write
\[
S=\{x\in\Z^n: x\in A^n, v-x\in A^n\}=\{x\in\Z^n: \Prb(x\in S_n,v-x\in S_n)>0\}.
\]
As before, $S$ is a finite set. From the above argument, there is no $x$ for which $x=v-x$.
Furthermore, the sets $\{x,v-x\}$ are all disjoint. Therefore
\begin{align*}
\Prb(v\in S_n+S_n)
&=1-\prod_{\substack{x\in S\\x<v-x}}(1-\Prb(x\in S_n,v-x\in S_n))\\
&=1-\prod_{\substack{x\in S\\x<v-x}}(1-\Prb(x\in S_n)\Prb(v-x\in S_n))\\
&=1-\prod_{\substack{x\in S\\x<v-x}}(1-\Prb(x\in S_n)\Prb(v-x\in T_n))\\
&=1-\prod_{\substack{x\in S\\x<v-x}}(1-\Prb(x\in S_n,v-x\in T_n))\\
&=\Prb(v\in S_n+^<T_n).
\end{align*}
As in the proof of Theorem~\ref{T:rainbow connection}, the result about expectations follows
by summing over $v\in R_n$.
\end{proof}

Theorems~\ref{T:rainbow connection} and~\ref{T:rainbow connection II} give us a good bound on $\E|S_n+k\cdot S_n|$
since, usually, most of the elements of $S_n+k\cdot T_n$ are in fact rainbow.
Recall that $(\E|S_n+k\cdot T_n|)^{1/n}\to\norm{\alpha+k\cdot\alpha}=\norm{\gamma}$ as $n\to\infty$.

\begin{thm}\label{T:taste the rainbow}
Let $\alpha$ be a fractional set with $\norm\alpha>1$, and let
$\gamma:=\alpha+k\cdot\alpha$ for a fixed nonzero integer $k$. For $n\ge 1$,
let $S_n$ and $T_n$ be drawn independently from $\alpha^n$. Then there exists an $\epsilon>0$ such that
$\E|(S_n+k\cdot T_n)\setminus R_n|=o((\norm\gamma-\epsilon)^n)$ as $n\to\infty$. 	
\end{thm}

\begin{proof} As above, write $\Gamma={\rm supp}(\gamma)$.
Let $p\in[0,1]$ be such that $\norm\gamma=\sum_{z\in\Gamma}\gamma(z)^p$, and let $\delta=\min_{z\in\Gamma}\gamma(z)$.
For $z\in\Gamma$, let
\[
Q_z=\{v\in S_n+k\cdot T_n: v_i\not= z{\rm \ for\ all\ }i\}.
\]
From the argument in the proof of Theorem~\ref{T:fractional2setupper}, we have, for all $0\le q\le 1$,
\[
\E|Q_z|\le \left(\sum_{w\ne z}\gamma(w)^q\right)^n.
\]
In particular, taking $q=p$,
\[
\E|Q_z|\le \left(\sum_{w\ne z}\gamma(w)^p\right)^n=(\norm\gamma-\delta^p)^n,
\]
and so
\[
\E|(S_n+k\cdot T_n)\setminus R_n|\le |\Gamma|(\norm\gamma-\delta^p)^n=o((\norm\gamma-\epsilon)^n,
\]
for any $\epsilon<\delta^p$.
\end{proof}

\begin{thm}\label{T:one set upper}
Let $\alpha$ be a fractional set with $\norm\alpha>1$, and let
$\gamma:=\alpha+k\cdot\alpha$ for a fixed nonzero integer $k$. For $n\ge 1$,
let $S_n$ be drawn from $\alpha^n$. Then
\[
\E|S_n+k\cdot S_n|\le2\norm{\gamma}^n.
\]
\end{thm}
\begin{proof}
Write
\[
S_n\,\hat{+}\,k\cdot S_n=\{a+kb: a,b\in S_n, a\not=b\}.
\]
Then
\[
S_n+k\cdot S_n=(S_n\,\hat{+}\,k\cdot S_n)\cup ((k+1)\cdot S_n)
\]
and so
\[
|S_n+k\cdot S_n|\le |S_n\,\hat{+}\,k\cdot S_n|+|S_n|.
\]
The argument in the proof of Theorem~\ref{T:fractional2setupper} shows that
\[
\E|S_n\,\hat{+}\,k\cdot S_n|\le\norm{\gamma}^n,
\]
and Lemmas~\ref{L:expectation} and~\ref{L:we do use this} show that $\E|S_n|=\norm{\alpha}^n\le\norm{\gamma}^n$, so that
\[
\E|S_n+k\cdot S_n|\le 2\norm{\gamma}^n,
\]
as required.
\end{proof}

\noindent{\it Proof of Theorem~\ref{T:fractional}.} Let $\alpha$ be a fractional set with $\norm{\alpha}> 1$, and suppose
$S_n\subseteq\Z^{n}$ is drawn from $\alpha^n$. We have $\E|S_n|=\norm{\alpha}^n$ and $\Var|S_n|\le\norm{\alpha}^n$
from Lemmas~\ref{L:expectation} and~\ref{L:variance} respectively. It remains to show that
\[
\lim_{n\to\infty}(\E|S_n+k\cdot S_n|)^{1/n}=\norm{\alpha+k\cdot\alpha}.
\]
Let $\gamma:=\alpha+k\cdot\alpha$. From Theorem~\ref{T:one set upper} we have, for all nonzero $k$,
\[
\limsup_{n\to\infty}\E|S_n+k\cdot S_n|^{1/n}\le\norm{\gamma}.
\]
Now, if $k\not=1$, Theorem~\ref{T:rainbow connection} gives
\begin{align*}
\E|S_n+k\cdot S_n|&\ge\E|(S_n+k\cdot S_n)\cap R_n|\\
&=\E|(S_n+k\cdot T_n)\cap R_n|\\
&=\E|S_n+k\cdot T_n|-\E|(S_n+k\cdot T_n)\setminus R_n|.
\end{align*}
By Theorem~\ref{T:fractional two}, we have $\E|S_n+k\cdot T_n|^{1/n}\to\norm{\gamma}$, and
Theorem~\ref{T:taste the rainbow} gives us an $\epsilon>0$ such that
$\E|(S_n+k\cdot T_n)\setminus R_n|=o((\norm{\gamma}-\epsilon)^n)$.
Consequently, we have
\[
\liminf_{n\to\infty}\E|S_n+k\cdot S_n|^{1/n}\ge\norm{\gamma}
\]
and hence $\E|S_n+k\cdot S_n|^{1/n}\to\norm{\gamma}$, as required.

\medskip

\noindent When $k=1$, Theorem~\ref{T:rainbow connection II} gives
\begin{align*}
\E|S_n+S_n|&\ge\E|(S_n+S_n)\cap R_n|\\
&=\E|(S_n+^<T_n)\cap R_n|\\
&=\E|S_n+^<T_n|-\E|(S_n+^<T_n)\setminus R_n|\\
&=\tfrac12\E|S_n\,\hat{+}\,T_n|-\E|(S_n+^<T_n)\setminus R_n|\\
&\ge\tfrac12(\E|S_n+T_n|-\E|S_n|)-\E|(S_n+^<T_n)\setminus R_n|\\
&=\tfrac12(\E|S_n+T_n|-\norm{\alpha}^n)-\E|(S_n+^<T_n)\setminus R_n|.
\end{align*}
Now, if $\norm{\gamma}=\norm{\alpha}$, we certainly have $\E|S_n+S_n|\ge\E|S_n|=\norm{\alpha}^n=\norm{\gamma}^n$,
so we may assume $\norm{\gamma}>\norm{\alpha}$. In this case, Theorems~\ref{T:fractional two}
and~\ref{T:taste the rainbow} now yield
\[
\liminf_{n\to\infty}\E|S_n+S_n|^{1/n}\ge\norm{\gamma}
\]
and hence $\E|S_n+S_n|^{1/n}\to \norm{\gamma}$, as required.

\begin{cor}
Let $\alpha$ be a fractional set with $\norm\alpha>1$, and let
$\gamma:=\alpha+k\cdot\alpha$ for $k\not=1$. For $n\ge 1$,
let $S_n$ be drawn from $\alpha^n$. Then,
if $\gamma$ is strictly spartan (i.e., $\sum_{n\in{\rm supp}(\gamma)}\gamma(n)\log\gamma(n)<0$),
then
\[
\E\left(|S_n|^2-|S_n+k\cdot S_n|\right)=o(\E|S_n|^2).
\]
\end{cor}

\begin{proof}
If $\gamma=\alpha+k\cdot\alpha$ is spartan, then $\norm{\gamma}=\norm{\alpha}^2$.
We have
\begin{align*}
\E|S_n+k\cdot S_n|
&\ge \E|S_n+k\cdot T_n|+o((\norm{\gamma}-\epsilon)^n)\\
&=\E|S_n||T_n|+o(\norm{\gamma}^n)+o((\norm{\gamma}-\epsilon)^n)\\
&=(\E|S_n|)^2+o(\norm{\gamma}^n)\\
&=\E(|S_n|^2)-\Var|S_n|+o(\norm{\gamma}^n)\\
&=\E(|S_n|^2)+o(\norm{\gamma}^n),
\end{align*}
where the first line follows from Theorems~\ref{T:rainbow connection}
and~\ref{T:taste the rainbow} (as in the proof of Theorem~\ref{T:fractional}),
the second line follows from Lemma~\ref{L:sparta for two sets},
and the last line is from Lemma~\ref{L:variance}. The conclusion follows, since the random
variable $|S_n|^2-|S_n+k\cdot S_n|$ is always nonnegative.
\end{proof}

\begin{cor}
Let $\alpha$ be a fractional set with $\norm\alpha>1$, and let
$\gamma:=\alpha+\alpha$. For $n\ge 1$,
let $S_n$ be drawn from $\alpha^n$. Then,
if $\gamma$ is strictly spartan (i.e., $\sum_{n\in{\rm supp}(\gamma)}\gamma(n)\log\gamma(n)<0$),
then
\[
\E\left(\tfrac12|S_n|^2-|S_n+S_n|\right)=o(\E|S_n|^2).
\]
\end{cor}

\begin{proof}
If $\alpha+\alpha$ is spartan, then $\norm{\gamma}=\norm{\alpha}^2$.
We have
\begin{align*}
\E|S_n+S_n|
&\ge\tfrac12\E|S_n+T_n|+o(\norm{\gamma}^n)\\
&=\tfrac12\E|S_n||T_n|+o(\norm{\gamma}^n)\\
&=\tfrac12(\E|S_n|)^2+o(\norm{\gamma}^n)\\
&=\tfrac12\E(|S_n|^2)-\tfrac12\Var|S_n|+o(\norm{\gamma}^n)\\
&=\tfrac12\E(|S_n|^2)+o(\norm{\gamma}^n),
\end{align*}
where the first line follows from Theorems~\ref{T:rainbow connection II}
and~\ref{T:taste the rainbow} (as in the proof of Theorem~\ref{T:fractional}),
the second line follows from Lemma~\ref{L:sparta for two sets},
and the last line is from Lemma~\ref{L:variance}. The conclusion follows, since the random
variable $\tfrac12|S_n|^2+\tfrac12|S_n|-|S_n+S_n|$ is always nonnegative, and $\E|S_n|=\norm{\alpha}^n
=o(\norm{\gamma}^n)$.
\end{proof}

Theorem~\ref{T:spartan} follows from the last two corollaries, Theorem~\ref{T:fractional},
and Markov's inequality,
noting that the random variables $|S_n|^2-|S_n+k\cdot S_n|$ (for $k\not=1$) and
$\tfrac12|S_n|^2+\tfrac12|S_n|-|S_n+S_n|$ are always nonnegative.

\subsection{Ruzsa's method, and Hennecart, Robert and Yudin's construction}\label{S:RuzsaHRY}

In~\cite{RuzsaRandom}, Ruzsa constructs sets by taking a finite set $S$, and a
fixed probability $0<q<1$, and selecting subsets of $\Z^n$ by taking each
element of $S^n$ independently with probability $q^n$. In our terminology, this is
the same as drawing from $\alpha^n$, where $\alpha$ is the fractional set
$q\mathbbm{1}_S$.

Let us suppose that $|S|=M$ and $|S+k\cdot S|=N$, and write
\[
S+k\cdot S=\{x_1,\ldots,x_N\}.
\]
For $1\le i\le N$, write $\lambda_i$ for the number of ordered pairs $(s_1,s_2)\in S^2$ such that
$x_i=s_1+ks_2$. We say that $\lambda_i$ is the {\em multiplicity} of $x_i$, and that the multiset
$\{\lambda_1, \lambda_2, \ldots, \lambda_N\}$ is the {\em multiplicity spectrum} of the fractional dilate
$S+k\cdot S$. Note that $\sum_{i=1}^N\lambda_i=M^2$. With these definitions, we have
\[
(\alpha+k\cdot\alpha)(x)=\begin{cases}
0 & \textrm{ if }x\notin S+k\cdot S,\\
q^2\lambda_i &\textrm{ if }x=x_i.
\end{cases}
\]
Theorem~\ref{T:fractionalsize} now specializes to the following.

\begin{thm}\label{T:Ruzsa} For a finite set $S\subseteq\Z$, and a fixed probability $0<q<1$,
let $\alpha=q\mathbbm{1}_S$. Then, for a nonzero integer $k$, with $M,N,\lambda_i$ defined
as above, and with $\gamma:=\alpha+k\cdot\alpha$, we have
\[
\norm{\gamma}=\norm{\alpha+k\cdot\alpha}=\begin{cases}
(qM)^2 &\textrm{ if }q^2\le\prod_i\lambda_i^{-\lambda_i/M^2}\ \ (\gamma{\rm\ is\ spartan})\\
N &\textrm{ if }q^2\ge\prod_i\lambda_i^{-1/N}\ \ (\gamma{\rm\ is\ opulent})\\
q^{2p}\sum_i\lambda_i^p &\textrm{ if }q^2=
\prod_i\lambda_i^{-\Lambda(p)}\ \ (\gamma {\rm\ is\ }p{\rm -comfortable}),\\
\end{cases}
\]
where
\[
\Lambda(p)=\frac{\lambda_i^p}{\sum_1^N\lambda_j^p}.
\]
\end{thm}
\noindent Before proceeding to the Hennecart, Robert and Yudin construction, we first
give an easier example, using the set $\{0, 1, 3, 7\}$.

\begin{clm}
For all $\epsilon>0$, there exists a set $S$ with
$|S-S|>|S|^{2-\epsilon}$ and $|S+S|<|S|^{1.8983+\epsilon}$.
\end{clm}
\noindent Using the proof of Corollary~\ref{C:1.7354} from the introduction, this will follow from
the next claim (just as Corollary~\ref{C:1.7354} follows from Theorem~\ref{T:1.7354}).

\begin{clm}
There exists a fractional set $\alpha$ for which $\norm\alpha>1$,
$\alpha-\alpha$ is strictly spartan (so that $\norm{\alpha-\alpha}=\norm{\alpha}^2$), and
$\norm{\alpha+\alpha}\le\norm{\alpha}^{1.8983}$.
\end{clm}
\begin{proof}
Let $S=\{0,1,3,7\}$, and let $\frac14<q<1$ be fixed. Then $S+S=\{0,2,6,14,1,3,7,4,8,10\}$, with corresponding
multiplicities $\{1,1,1,1,2,2,2,2,2,2\}$. From Theorem~\ref{T:Ruzsa}, we have
\[
\norm{\alpha+\alpha}=\begin{cases}
(4q)^2 &\textrm{ if }q\le 2^{-\frac38},\\
10 &\textrm{ if }q\ge 2^{-\frac{3}{10}},\\
2q^{2p}(2+3\cdot2^p) &\textrm{ if }q=2^{-\frac{3\cdot 2^p}{2(2+3\cdot2^p)}}.\\
\end{cases}
\]
On the other hand, we have that $S-S=\{-7,-6,-4,-3,-2,-1,0,1,2,3,4,6,7\}$, with corresponding multiplicities
$\{1,1,1,1,1,1,4,1,1,1,1,1,1\}$, so that
\[
\norm{\alpha-\alpha}=\begin{cases}
(4q)^2 &\textrm{ if }q\le 2^{-\frac14},\\
13 &\textrm{ if }q\ge 2^{-\frac1{13}},\\
q^{2p}(12+4^p) &\textrm{ if }q=2^{-\frac{4^p}{12+4^p}}.\\
\end{cases}
\]
Note that
\[
\tfrac14<2^{-\frac{3}{10}}<2^{-\frac14},
\]
so that if we let $q$ be just slightly less than $2^{-\frac14}$, $\alpha-\alpha$ will be strictly spartan,
and also
\[
\norm{\alpha}=4q=2^{\frac74}-\epsilon,\ \ \ \ \ \norm{\alpha+\alpha}=10,\ \ \ \ \ \norm{\alpha-\alpha}=(4q)^2=\left(2^{\frac74}-\epsilon\right)^2,
\]
so that
\[
\norm{\alpha+\alpha}<\norm{\alpha}^{\frac{4\log 10}{7\log 2}+\epsilon'}<\norm{\alpha}^{1.8983},
\]
as required.
\end{proof}

Hennecart, Robert and Yudin~\cite{HRY} constructed
sets $A_{k,d}\in\Z^{d+1}$ for which $|A_{k,d}-A_{k,d}|$ is much larger than
$|A_{k,d}+A_{k,d}|$. Their construction is
\[
A_{k,d}=\{(x_1, \ldots, x_{d+1}):
x_1, \ldots, x_{d+1}\ge 0,\,x_1+\cdots+x_{d+1}= k\}.
\]
A standard textbook argument shows that $|A_{k,d}|=\binom{k+d}{d}$.
It turns out that the multiplicity spectrum for  $A_{k,d}-A_{k,d}$ is particularly easy to describe.

\begin{thm}\label{T:HNY Subtraction}
The multiplicity spectrum for $A_{k,d}-A_{k,d}$ consists of

\medskip

$\bullet$ one copy of $\binom{k+d}d$

\smallskip

$\bullet$ $\sum_{i=1}^{\min(t,d)}\binom{d+1}i\binom{t-1}{i-1}\binom{t+d-i}{d-i}$ copies of
$\binom{k+d-t}d$, for $1\le t\le k$.

\medskip
\end{thm}

\begin{proof}
Clearly, the multiplicity of $0\in A_{k,d}-A_{k,d}$ is $|A_{k,d}|=\binom{k+d}{d}$.

For the other multiplicities, fix $0\not=w=(w_1, \ldots, w_{d+1})\in A_{k,d}-A_{k,d}$. Then $w=x-y$,
where $x=(x_1, \ldots, x_{d+1})$ and $y=(y_1, \ldots,
y_{d+1})$ are nonnegative vectors summing to $k$. It follows that $\sum_iw_i=0$.

Next, write $w^{+}$ for the vector containing only the positive coordinates of $w$, so that
\[
w^{+}=(\max\set{w_1, 0}, \ldots, \max\set{w_{d+1},0}),
\]
and write $w^{-}$ for the corresponding vector with the negative coordinates. Then,
for all $i$, we have
$\max\set{w_i, 0}\le x_i$, so that $1\le\sum_i(w^{+})_i:=t\le k$.
If $w$ is any integer vector with $\sum_i w_i=0$, and $x$ and $y$
are nonnegative vectors such that $w=x-y$, then the vector $z=x-w^{+}$
is a nonnegative vector with $\sum_i z_i=k-t$. Conversely, for any such $z$,
the vectors $x=w^{+}+z$ and $y=z-w^{-}$ are nonnegative vectors, both
summing to $k$, and satisfying $w=x-y$. Consequently,
the multiplicity of $w$ in $A_{k,d}-A_{k,d}$ is just the number of choices for $z$,
i.e., the number of nonnegative $(d+1)$-dimensional vectors summing to
$k-t$. This number is $\binom{k-t+d}{d}$.

To calculate the number of integer vectors of length $d+1$, summing to 0,
whose positive elements sum to $t>0$, we classify such vectors according to
the number $i$ of their positive elements. This number must be in the range $1\le i\le\min\set{t,d}$.
The number of choices for the locations of the $i$ positive elements is $\binom{d+1}{i}$.
The number of choices for the $i$ positive numbers summing to $t$ equals
the number of choices for $i$ nonnegative numbers summing to $t-i$, which is $\binom{t-1}{i-1}$.
Finally, the number of choices for the $d+1-i$ nonpositive numbers summing to $-t$ is
$\binom{t+d-i}{d-i}$.
\end{proof}

This allows us to prove Theorem~\ref{T:1.7354}, namely, that there exists a fractional
set $\alpha$ for which $\norm\alpha>1$,
$\alpha-\alpha$ is strictly spartan, and with
$\norm{\alpha+\alpha}\le\norm{\alpha}^{1.7354}$.

\begin{proof}[Proof of Theorem~\ref{T:1.7354}]
Our $\alpha$ will be $q\mathbbm{1}_{A_{k,d}}$, where $k$ and $d$ will be chosen later, and
where $0<q<1$ is a probability that will be chosen in terms of $k$ and $d$.
For $0\le t\le k$, let $\lambda_t=\binom{k+d-t}d$. Let $\mu_0=1$, and, for $t>0$, let
\[\mu_t=\sum_{i=1}^t\binom{d+1}i\binom{t-1}{i-1}\binom{t+d-i}{d-i}.
\]

\noindent Now, from Theorem~\ref{T:HNY Subtraction},
we know that the nonzero values of $\alpha-\alpha$ consist of
$\mu_t$ copies of $q^2\lambda_t$, for each $0\le t\le k$. Write
$A=A_{k,d}-A_{k,d}$. Then, from Theorem~\ref{T:fractionalsize}, $\alpha-\alpha$ is strictly spartan when
\[
\sum_{x\in A}(\alpha-\alpha)(x)\log_2 (\alpha-\alpha)(x)<0,\\
\]
or
\[
\sum_{t=0}^k q^2\lambda_t\mu_t\log_2(q^2\lambda_t)<0.
\]
Rearranging, we need
\[
\sum_{t=0}^{k}2\lambda_t\mu_t\log_2 q<
-\sum_{t=0}^k\lambda_t\mu_t\log_2\lambda_t.
\]
In summary, for $\alpha-\alpha$ to be strictly spartan, we need $q<2^{-f(k,d)}$, where
\[
f(k,d)=\frac{\sum_{t=0}^k\lambda_t\mu_t\log_2(\lambda_t)}{2\sum_{t=0}^{k}\lambda_t\mu_t}.
\]

\noindent We need to show that we can choose $q$ so that, in addition, $\norm{\alpha}=q\binom{k+d}{d}>1$.
For this, we need an upper bound on $f(k,d)$. But $2f(k,d)$ is just a weighted average of the
$k+1$ numbers $\log_2\lambda_t$. Therefore
%\begin{align*}
%2f(k,d)&=\frac{\sum_{t=0}^k\mu_t\lambda_t\log_2\lambda_t}{\sum_{t=0}^{k}\mu_t\lambda_t}\\
%&\le \max_t\log_2\lambda_t=\log_2\lambda_0\\
%&=\log_2\binom{k+d}{d}.
%\end{align*}
\[
2f(k,d)=\frac{\sum_{t=0}^k\mu_t\lambda_t\log_2\lambda_t}{\sum_{t=0}^{k}\mu_t\lambda_t}
\le \max_t\log_2\lambda_t=\log_2\lambda_0
=\log_2\binom{k+d}{d}.
\]

\noindent Consequently, we have $2^{-f(k,d)}\ge\binom{k+d}{d}^{-1/2}$, so that if we choose $q$ satisfying
\[
\binom{k+d}{d}^{-1}<q<\binom{k+d}{d}^{-1/2}\le 2^{-f(k,d)}=:p(k,d)
\]
then $\alpha-\alpha$ will be strictly spartan, and we will also have $\norm{\alpha}>1$.

Finally, we turn to $\alpha+\alpha$. We have
\[
A_{k,d}+A_{k,d}= \{(x_1, \ldots, x_{d+1}):
x_1, \ldots, x_{d+1}\ge 0,\,x_1+\cdots+x_{d+1}= 2k\},
\]
so $\norm{\alpha+\alpha}\le|A_{k,d}+A_{k,d}|=\binom{2k+d}{d}$. Note that we believe
$\alpha+\alpha$ is opulent, so we have equality here, but we do not need that for this proof.
If we let
\[
\beta=\beta(k,d)=\frac{\log\binom{2k+d}{d}}{\log(p(k,d)|A_{k,d}|)},
\] it follows that
$\norm{\alpha+\alpha}\le\norm{\alpha}^\beta$. Computer calculations show that the function
$\beta(k,d)$ seems to have a global minimum of $\beta=1.735383\ldots$ at $d=14929$ and $k=987$.
\end{proof}

\subsection{MSTD sets and the region $F_{1,-1}$}

In this subsection, we collect all the results displayed in Figure~\ref{fig:feas2}.

For a finite set $A\subset Z$, with sumset $A+A$ and difference set $A-A$, write
\[
|A+A|=|A|^x{\rm\ \ \ and\ \ \ }|A-A|=|A|^y.
\]
In 1973, Freiman and Pigarev~\cite{FP} proved
\[
|A+A|^{3/4}\le |A-A|\le |A+A|^{4/3} {\rm \ \ \ or\ \ \ } \tfrac34 x\le y\le \tfrac43 x.\eqno(1)
\]
Noting that $1\le x,y\le 2$, this is weaker than the best known bounds, proved by Ruzsa~\cite{Ruz1976,RuzsaRandom}:
\[
\left(\frac{|A+A|}{|A|}\right)^{\frac12}\le \frac{|A-A|}{|A|}\le \left(\frac{|A+A|}{|A|}\right)^2 {\rm \ \ \ or\ \ \ } y\le 2x-1 {\rm\ \ \ and\ \ \ }x\le 2y-1.\eqno(2)
\]
The upper bound comes from taking $B=C=-A$ in Ruzsa's triangle inequality:
\[
|A||B-C|\le |A-B||A-C|,
\]
and the lower bound comes from taking $B=C=-A$ in Corollary 7.3.6 of~\cite{Zhao}:
\[
|A||B+C|\le |A+B||A+C|.
\]
Together, these results show that the shaded regions in Figure~\ref{fig:feas2} are all infeasible.

Next we turn to MSTD sets; these are sets $A$ satisfying $|A+A|>|A-A|$.
The terminology is due to Nathanson, and stands for ``more sums than differences".
According to Nathanson (see Footnote 1 of~\cite{Nat}), Conway found the first MSTD set in 1967:
\[
A=\{0,2,3,4,7,11,12,14\}.
\]
This is the smallest MSTD set, it's unique among sets of size 8 up to scaling and shifting, and it satisfies
\[
|A|=8\ \ \ \ \ |A-A|=25\ \ \ \ \ |A+A|=26.
\]
Various families of MSTD sets are known. The record-breaker, in terms of minimizing both $\frac{y}{x}$ and $\frac{y-1}{x-1}$
(in our notation), is due to Penman and Wells~\cite{PW}. They construct a family of sets $A_t$, for $t\ge 1$, with
\[
|A_t|=5t+17\ \ \ \ \ |A_t-A_t|=26t+61\ \ \ \ \ |A_t+A_t|=32t+63.
\]
Both these constructions are also illustrated in Figure~\ref{fig:feas2}, with the latter one plotted for real $t\ge1$.

Finally, we discuss the construction of Hennecart, Robert and Yudin~\cite{HRY}. As already mentioned, their
construction is
\[
A_{k,d}=\{(x_1, \ldots, x_{d+1}):
x_1, \ldots, x_{d+1}\ge 0,\,x_1+\cdots+x_{d+1}= k\},
\]
and they show that
\[
|A|=\binom{k+d}{d}\ \ \ \ \ \ \ \ |A+A|=\binom{2k+d}{d}\ \ \ \ \ \ \ \ |A-A|=\sum_{t=0}^{\min(d,k)}\binom{d}{t}^2\binom{k+d-t}{d}.
\]
Taking $k=ad$, for $a>0$, yields the curve shown in Figure~\ref{fig:feas2}. Hennecart, Robert and Yudin choose $k=d/2$,
so that
\[
|A|=\binom{\tfrac{3d}{2}}{d}\ \ \ \ \ \ \ \ |A+A|=\binom{2d}{d}\ \ \ \ \ \ \ \ |A-A|\ge\max_{0\le t\le \tfrac{d}{2}}\binom{d}{t}^2\binom{\tfrac{3d}{2}-t}{d}.
\]
This results in the asymptotics
\[
\frac{\log|A+A|}{\log|A|}\to\frac{4\log 2}{3\log 3-2\log 2}\approx 1.4520
\]
and
\[
\frac{\log|A-A|}{\log|A|}\to\frac{4\log(1+\sqrt2)}{3\log 3-2\log 2}\approx1.8463,
\]
leading to the point $(1.4520,18463)$ plotted in Figure~\ref{fig:feas2}.

\section{Open Questions}\label{S:Open Questions}

We introduced the concept of a fractional dilate as a more general version of
Ruzsa's method from ~\cite{RuzsaRandom}. Ruzsa's method is the
specialisation where a fractional dilate has all nonzero values equal.
However, we in fact only used this same specialisation to prove
Theorem~\ref{T:1.7354}. We believe that using a more general fractional dilate
supported on the Hennecart, Robert and Yudin sets $A_{k,d}$ could give a better
bound than 1.7354. In particular, the best bound one can get
from a Ruzsa-style dilate on $A_{2,d}$ is that for $d=23$, which with a value of
\[
q=5^{-\frac1{300}}2^{-\frac{46}{625}}12^{-\frac{1129}{15000}}
\]
yields a bound of 1.7897. However, if one takes a fractional dilate on
$A_{2,22}$ with a value of approximately $0.9951$ on the elements of the form
$2e_i$, and approximately $0.7617$ on the elements of the form $e_i+e_j$, one can instead
get a bound of 1.7889.

Our proof of Theorem~\ref{T:1.7354} relies on the fact that if $S_n$ is drawn
from $\alpha^n$, and $\alpha-\alpha$ is spartan, then the size of $S_n-S_n$ is
close to its expected value. We believe that this holds without the requirement
of spartaneity.

\begin{conj}\label{C:Big Conjecture}
If $\alpha$ is a fractional dilate with $\norm\alpha>1$, $k$ is a positive
integer, and $A_n$ is drawn from $\alpha^n$, then
\[
\lim_{n\to\infty}\frac{\log |A_n+k\cdot A_n|}{n}=\log\,\norm{\alpha+k\cdot\alpha}.
\]
\end{conj}

This would directly imply that various results for the sizes of sums and differences
of sets would also hold for fractional dilates. For example, Ruzsa's Triangle
Inequality would imply that
\[
\norm\alpha\norm{\beta-\gamma}\le\norm{\alpha-\beta}\norm{\alpha-\gamma}
\]
for fractional dilates $\alpha, \beta, \gamma$.

A weaker conjecture is that the feasible regions for dilates coincide with
the feasible regions for fractional dilates.

\begin{conj}\label{C:Small Conjecture}
For any fractional dilate $\alpha$, any positive integer $N$, and any
$\epsilon>0$, there exists a finite subset $S$ of the integers such that,
for all $|k|\le N$,
\[
\left|\frac{\log\norm{\alpha+k\cdot\alpha}}{\log\norm\alpha}-
\frac{\log|S+k\cdot S|}{\log|S|}\right|<\epsilon.
\]
\end{conj}

We also ask whether the fractional dilate versions of the open questions from
Section~\ref{S:Construction for (1,2)} are true. For the reader's convenience,
we write out these questions in full. Note that Theorem~\ref{T:fractionalsize}
gives a useful way of computing $\norm{\alpha+\alpha}$ and $\norm{\alpha+2\cdot\alpha}$.

\begin{qn}\label{Q:questions for dilates}
Suppose that $\alpha:\Z\to[0,1]$ is a function with finite support. We write
\begin{align*}
\norm\alpha&=\sum_i\alpha(i)\\
\norm{\alpha+\alpha}&=\inf_{0\le p\le 1}\sum_i\left(\sum_{j+k=i}\alpha(j)\alpha(k)\right)^p\\
\norm{\alpha+2\cdot\alpha}
&=\inf_{0\le p\le 1}\sum_i\left(\sum_{j+2k=i}\alpha(j)\alpha(k)\right)^p.
	\end{align*}

\noindent Are each of the following statements true for all such $\alpha$?

\begin{enumerate}
\item $\norm\alpha\norm{\alpha+2\cdot\alpha}\le\norm{\alpha+\alpha}^2$
\item $\norm{\alpha+2\cdot\alpha}\le\log_34\norm{\alpha+\alpha}$
\item $\norm{\alpha+2\cdot\alpha}\ge\norm{\alpha+\alpha}$
\end{enumerate}
\end{qn}

Since a subset of the integers is just a fractional dilate of its
characteristic function, positive answers to these questions would imply
positive answers to the corresponding questions in
Section~\ref{S:Construction for (1,2)}. If either
Conjecture~\ref{C:Big Conjecture} or Conjecture~\ref{C:Small Conjecture}
is true, the questions in the two sections are equivalent. A negative answer to
any of these questions would either lead to an extension of the feasible
region $F_{1,2}$, or to a better understanding of the above conjectures.

The biggest open question we leave is whether fractional dilates can find a
use elsewhere.

\bibliographystyle{amsplain}
\bibliography{CPS}

%\printbibliography
\end{document}